\documentclass[a4paper,11pt]{amsart}

\usepackage[utf8]{inputenc}

\usepackage{amsmath}
\usepackage{amsthm} 
\usepackage{amssymb} 
\usepackage{amsfonts}
\usepackage{enumerate,amscd,amsxtra,MnSymbol}
\usepackage{mathrsfs}
\usepackage{bbm}

\usepackage{hyperref} 
\usepackage[capitalise]{cleveref} 

\usepackage{geometry}
\usepackage{color}
\usepackage[cmtip,all]{xy}
\usepackage{tikz} 
\usepackage{tikz-cd} 
\usepackage{tikz-qtree} 
\usepackage{forest}
\usetikzlibrary{arrows,calc}

\usepackage{graphicx}
\usepackage{xparse}

\usepackage[normalem]{ulem}

\newenvironment{tz}{\begin{center}\begin{tikzpicture}}{\end{tikzpicture}\end{center}}

\tikzstyle{d}=[double distance=.3ex]
\tikzstyle{w}=[preaction={draw=white, -,line width=4pt}]
\NewDocumentCommand{\punctuation}{ m m O{5pt} }{\node at ($(#1.east)-(0,#3)$) {#2};}

\newcounter{diagram}
\renewcommand{\thediagram}{\thetheorem}

\tikzset{%
node distance=1.5cm, la/.style={scale=0.8}, rr/.style={xshift=1.5cm},
space/.style={xshift=.5cm}, over/.style={auto=false,fill=white,inner sep=1.5pt, minimum size=0, outer sep=0},
    symbol/.style={%
        draw=none,
        every to/.append style={%
            edge node={node [sloped, allow upside down, auto=false]{$#1$}}},
            
    }, pro/.style={postaction={decorate,decoration={
        markings,
        mark=at position .5 with {\node at (0,0) {$\bullet$};}
      }},
      inner sep=.9ex,
      },
  n/.style={double equal sign distance, -implies}, t/.style={double distance=2.5pt, -implies, postaction={draw,-}},
}

\usepackage{mathtools}

\theoremstyle{definition}

\newtheorem{theorem}{Theorem}[section]
\newtheorem*{theorem*}{Theorem}
\newtheorem{lemma}[theorem]{Lemma}
\newtheorem{proposition}[theorem]{Proposition}
\newtheorem*{proposition*}{Proposition}
\newtheorem{corollary}[theorem]{Corollary}
\newtheorem*{corollary*}{Corollary}

\newtheorem{construction}[theorem]{Construction}
\newtheorem{example}[theorem]{Example}

\newtheorem{definition*}{Definition}

\newtheorem{conjecture*}{Conjecture}

\newtheorem{notation}[theorem]{Notation}

\newtheorem{remark}[theorem]{Remark}
\newtheorem{remark*}{Remark}

\newtheorem{definition}[theorem]{Definition}

\newtheorem{introtheorem}{Theorem}

\newcommand{\caR}{{\mathcal R}}
\newcommand{\caC}{{\mathcal C}}
\newcommand{\caD}{{\mathcal D}}
\newcommand{\caB}{{\mathcal B}}
\newcommand{\caA}{{\mathcal A}}
\newcommand{\caE}{{\mathcal E}}

\newcommand{\caM}{{\mathcal M}}

\newcommand{\caS}{{\mathcal S}}
\newcommand{\caO}{{\mathcal O}}
\newcommand{\caH}{{\mathcal H}}

\newcommand{\caP}{{\mathcal P}}

\newcommand{\caQ}{{\mathcal Q}}

\newcommand{\caW}{{\mathcal W}}
\newcommand{\caF}{{\mathcal F}}
\newcommand{\caV}{{\mathcal V}}

\newcommand{\Mor}{{\mathrm{Mor}}}

\newcommand{\bE}{{\mathbb{E}}}

\newcommand{\fib}{\mathsf{fib}}

\newcommand{\op}{\mathsf{op}}
\newcommand{\dual}{\vee}

\newcommand{\Hom}{\mathsf{Hom}}

\newcommand{\Sp}{\mathsf{Sp}}

\newcommand{\X}{\mathrm{X}}

\newcommand{\GW}{\mathsf{GW}}

\newcommand{\Sym}{\mathsf{Sym}}

\newcommand{\id}{\mathsf{id}}

\newcommand{\Cat}{\mathsf{Cat}}

\newcommand{\Set}{\mathsf{Set}}

\newcommand{\map}{\mathrm{map}}

\newcommand{\Spc}{\mathsf{Spc}}

\newcommand{\Add}{\mathsf{Add}}
\newcommand{\St}{\mathsf{St}}
\newcommand{\Mon}{\mathsf{Mon}}

\newcommand{\Fun}{\mathsf{Fun}}

\newcommand{\Grp}{\mathsf{Grp}}

\newcommand{\Fin}{\mathsf{Fin}}

\newcommand{\Ar}{\mathsf{Ar}}

\newcommand{\Tw}{\mathsf{Tw}}

\newcommand{\Span}{\mathsf{Span}}
\newcommand{\KR}{\mathsf{KR}}

\newcommand{\Wald}{\mathsf{Wald}}

\newcommand{\gd}{\mathsf{gd}}
\newcommand{\Exact}{\mathsf{Exact}}
\renewcommand{\d}{\mathsf{d}}

\newcommand{\ev}{\mathsf{ev}}

\newcommand{\slice}[2]{{#1}_{/{#2}}}

\title{An equivalence between two frameworks for real algebraic $K$-theory}

\author[H.\ Heine]{Hadrian Heine}
\address{Max Planck Institute for Mathematics, Bonn, Germany}
\email{heine@mpim-bonn.mpg.de}

\author[M.\ Spitzweck]{Markus Spitzweck}
\address{Fakult\"at f\"ur Mathematik, Universit\"at Osnabr\"uck, Osnabr\"uck, Germany}
\email{markus.spitzweck@uni-osnabrueck.de}

\author[P.\ Verdugo]{Paula Verdugo}
\address{Max Planck Institute for Mathematics, Bonn, Germany}
\email{verdugo@mpim-bonn.mpg.de}

\begin{document}

\begin{abstract}
We prove an equivalence between the real $K$-theory genuine $C_2$-spectra of \cite{Calmes_etal2} for Poincar\'e $\infty$-categories and the one of \cite{realKthHSV} for Waldhausen $\infty$-categories with genuine duality.
\end{abstract}

\maketitle

\tableofcontents

\section{Introduction}
Hermitian $K$-theory originated from the problem of classifying hermitian forms and became a powerful invariant connecting different areas of mathematics, such as algebraic surgery theory, arithmetic geometry, and motivic homotopy theory. In the classical approaches to hermitian $K$-theory it was necessary to 
assume that 2 is invertible in the ring. This condition guarantees that hermitian and quadratic forms over the ring are equivalent, which was crucial to make hermitian $K$-theory well-behaved and amenable for computations. For example \cite{schlichting2017hermitian} achieved a nonconnective delooping of the Grothendieck-Witt space to a Grothendieck-Witt spectrum only if 2 is invertible in the ring.
Hornbostel \cite{Hornbostel} proved that hermitian $K$-theory is representable by a motivic spectrum.

This was a motivation to look for new foundational frameworks for hermitian $K$-theory that are well-behaved even when 2 is not invertible in the ring. Recently, several new such frameworks for hermitian $K$-theory have been developed in different contexts. Namely, that of Calm\`es-Dotto-Harpaz-Hebestreit-Land-Moi-Nardin-Nikolaus-Steimle
 \cite{Calmes_etal1,Calmes_etal2,Calmes_etal3}, that of Schliching \cite{schlichting1}, and that of the authors of this work \cite{realKthHSV}. 

In \cite{Calmes_etal1,Calmes_etal2,Calmes_etal3} they develop a framework in the realm of stable $\infty$-categories. Here the input objects are Poincar\'e categories, that is, stable $\infty$-categories equipped with a quadratic functor\textemdash a categorification of quadratic form proposed by Lurie \cite{lurie.L4}. A notable feature of this theory is that it builds a bridge between hermitian $K$-theory and $L$-theory
and refines work of Ranicki \cite{ranicki1980algebraic} in algebraic surgery theory.

\cite{schlichting1} builds a theory in the setting of so called
form categories with weak equivalences, exact categories with weak equivalences and duality equipped with a suitable quadratic structure.
In \cite{realKthHSV} we develop a framework in the setting of Waldhausen $\infty$-categories in the sense of Barwick \cite{barwick.wald}.
We construct real $K$-theory for Waldhausen $\infty$-categories with genuine duality\textemdash that is, Waldhausen $\infty$-categories equipped with a
duality and compatible quadratic data.

Marlowe-Schlichting compare in \cite{Schlichting2} the
framework of \cite{schlichting1} with the one of Calm\`es et al.
In this note, we provide a comparison between our framework and the one of
Calm\`es et al.

In \cite{Calmes_etal1,Calmes_etal2,Calmes_etal3} they construct hermitian $K$-theory via a hermitian $Q$-construction \cite{Calmes_etal2} that turns a Poincar\'e $\infty$-category to a complete Segal object of Poincar\'e $\infty$-categories. 
In our approach \cite{realKthHSV} we use a real $S$-construction,
a variant of the hermitian $S$-construction of Schlichting \cite[2.10]{schlichting.mv}
and of the real Segal construction of Hesselholt-Madsen \cite[Definition 5.1.]{hesselholt-madsen}. Our $S$-construction turns a Waldhausen $\infty$-category with genuine duality to a real Segal object of Waldhausen $\infty$-categories with genuine duality.

In this note we compare these two constructions; our main result is as follows.

\begin{introtheorem}[\cref{theorem_compsp}]\label{thm_spectra_intro}
For every stable $\infty$-category $\caC$ with genuine duality and associated Poincar\'e $\infty$-category $\caC'$
there is an equivalence of genuine $C_2$-spectra 
$$\KR(\caC) \to \KR'(\caC')$$
between the real $K$-theory spectrum of \cite[Definition 9.21]{realKthHSV} and the real $K$-theory spectrum of \cite[Definition 4.5.1]{Calmes_etal2}.
\end{introtheorem}

The seed for the comparison between the aforementioned frameworks was presented in \cite[Section 6.3]{realKthHSV}. There we establish an equivalence of $\infty$-categories that can be summarized in the  following sentence:

\vspace{0.5em}
\begin{center}
\emph{Stable $\infty$-categories with genuine duality are precisely quadratic functors.}
\end{center}
\vspace{0.5em}

As the first step we extend the hermitian $Q$-construction of \cite{Calmes_etal2} to the realm of Waldhausen $\infty$-categories with genuine duality (\cref{hermq}) and compare it to our real version of $S$-construction \cite{realKthHSV}. This comparison is the content of \cref{comp}, presented below.

\begin{introtheorem}[\cref{comp}]\label{theorem_SQ} 
Let $\caC$ be a Waldhausen $\infty$-category with genuine duality. The map $$\theta_\caC\colon S(\caC)\circ e \to \caQ(\caC)$$ of simplicial Waldhausen $\infty$-categories with genuine duality is an equivalence.
\end{introtheorem}

Since the Grothendieck-Witt spaces of \cite{Calmes_etal2} and \cite{realKthHSV} are constructed via the hermitian $Q$-construction and real $S$-construction, respectively, we obtain the following result.

\begin{introtheorem}[\cref{cor:Ktheory_coincide}]\label{theorem_GW}
Let $\caC$ be a stable $\infty$-category with genuine duality identified with a Poincar\'e $\infty$-category. There is a canonical equivalence
between the Grothendieck-Witt spaces of \cite{Calmes_etal2} and \cite{realKthHSV}.
\end{introtheorem}

We use \cref{theorem_GW} as a stepping stone for our main result, \cref{thm_spectra_intro}. In an earlier version of this note \cite{comparisonKth}, we had claimed an equivalence between \emph{real} $K$-theory spaces; our proof at the time was insufficient for that, but sufficient to obtain this equivalence between hermitian $K$-theory spaces as stated\textemdash the originally claimed equivalence between real $K$-theory spaces is now a direct consequence of \cref{thm_spectra_intro}. In the interim, \cref{theorem_GW} was independently proven by Calm\`es et al. and is now found in \cite[Proposition B.1.2]{Calmes_etal2}.

\subsection*{Acknowledgments}
The third named author gratefully acknowledges the support of an international Macquarie University Research Excellence Scholarship during part of the development of this work. The authors also thank the Max Planck Institute for Mathematics in Bonn for its hospitality for a visit of the first two named authors to the third.

\section{Background}

This work is not meant to be self-contained. We do, however, recall some of the definitions and constructions that are used and provide references for other results. Throughout this document, we add the word ``real'' to mean $\Spc^{C_2}$-enriched.

\subsection{Input}

In this section we will recall the definition of Waldhausen $\infty$-categories with genuine duality, which will be the input of our real K-theory functor.

\subsection*{Waldhausen $\infty$-categories with duality}\label{subsec:wald_duality}

In this subsection we make a brief recap of Waldhausen $\infty$-categories as presented by Barwick (see \cite{barwick.wald} for details), and define exact $\infty$-categories (\cref{def:exact_infty_cats}) which will easily encode (plain) duality.

\begin{definition}
A Waldhausen $\infty$-category is a pair $(\caD, \caC) $ with $\caD$ an  $\infty$-category with zero object and $\caC$ a subcategory of $\caD$, whose morphisms are called cofibrations, subject to the following conditions.

\begin{itemize}
\item Every object $X$ of $\caD$ is cofibrant, i.e. the unique map $0 \to X$ is a cofibration.
	
\item The cobase change in $\caD$ of any cofibration exists and is again a cofibration.
\end{itemize}

A map of Waldhausen $\infty$-categories $(\caD, \caC) \to (\caD', \caC') $
is a functor $\caD \to \caD'$ that preserves the initial object, cofibrations and pushouts along cofibrations.
\end{definition}

We will denote by $\Wald_\infty \subset \Fun([1], \Cat_\infty)$ the subcategory spanned by the Waldhausen $\infty$-categories and maps between them as defined above. 

\begin{remark}
The axioms of a Waldhausen $\infty$-category imply that every equivalence is a cofibration.    
\end{remark}

It will relevant to note that the notion of Waldhausen $\infty$-category is not self dual. We call a pair $(\caC,\caF)$ a co-Waldhausen $\infty$-category if $(\caC^\op,\caF^\op)$ is a Waldhausen $\infty$-category. Next we define exact $\infty$-categories, first presented in this in \cite{barwick.exact}.

\begin{definition}Let $(\caD,\caC,\caF)$ be a triple where $\caD$ is an $\infty$-category and $\caC,\caF$ are subcategories, whose morphisms we call cofibrations and fibrations respectively. A commutative square $\sigma$ in $\caD$
\[
\begin{tikzcd}
A \ar[r, "f"] \ar[d, "\alpha"] & B \ar[d, "\beta"] \\
C \ar[r, "g"]  & D
\end{tikzcd}
\]

\begin{itemize}
\item is an ambigressive pushout square if $\sigma$ is a pushout square, $f$ is a cofibration and $\alpha$ is a fibration, and 
\item an ambigressive pullback square if $\sigma$ is a pullback square, $g$ is a cofibration and $\beta$ is a fibration.
\end{itemize}

We call a square as above exact when is both an ambigressive pushout and an ambigressive pullback.
\end{definition}

\begin{notation}
Let $\mathrm{Amb}(\caD, \caC, \caF) \subset \Fun([1]\times[1], \caD)$ be the full subcategory of exact squares.
    
\end{notation}

\begin{definition}\label{def:exact_infty_cats}
An exact $\infty$-category is a triple $(\caD, \caC, \caF)$ with $\caD$ an additive $\infty$-category and $\caC, \caF$ subcategories of $\caD$, whose morphisms are called respectively cofibrations and fibrations, subject to the following conditions.

\begin{enumerate}
\item The pair $(\caD, \caC) $ is a Waldhausen $\infty$-category.
\item The pair $(\caD, \caF) $ is a co-Waldhausen $\infty$-category.
\item A commutative square in $\caD$ is an ambigressive pushout square if and only if it is an ambigressive pullback square.
\end{enumerate}
\end{definition}

In light of condition 3, in an exact $\infty$-category we call ambigressive pushout squares respectively ambigressive pullback squares simply ambigressive squares.

\begin{definition}
Given exact $\infty$-categories $(\caD, \caC, \caF), (\caD', \caC', \caF') $
a functor $F\colon \caD \to \caD'$ is called exact (with respect to the exact structures on $\caC$ and $\caD$) if $F$ preserves the zero object, cofibrations, fibrations, pushouts along cofibrations and pullbacks along fibrations.
\end{definition}

Denote by $\Exact_\infty \subset \Fun(\Lambda^2_2, \Cat_\infty)$ the subcategory spanned by the exact $\infty$-categories and exact functors.

\begin{remark} The non-trivial $C_2$-actions on $\Lambda^2_2$ and $\Cat_\infty$ induce a 
$C_2$-action on the functor $\infty$-category 
$\Fun(\Lambda^2_2, \Cat_\infty)$ that restricts to $\Exact_\infty$. 
\end{remark}

\begin{definition}
We define the $\infty$-category of small Waldhausen $\infty$-categories with duality as the homotopy fixed points of the action on $\Exact_\infty$ described above. We write $$\Wald_\infty^\d\coloneqq \Exact_\infty^{hC_2}.$$
\end{definition}

In other words, a Waldhausen $\infty$-category with duality is a Waldhausen $\infty$-category, whose underlying $\infty$-category carries a duality such that the cofibrations together with the opposites of the cofibrations form an exact $\infty$-category. 

Relevant examples of these kind of categories are given by stable $\infty$-categories with duality.

\subsection*{Waldhausen $\infty$-categories with genuine duality}

\begin{definition}\label{Waldgd}
A small Waldhausen $\infty$-category with genuine duality is 
a pair $(\caE, \phi)$, with $E$ a small Waldhausen $\infty$-category with duality
and $\phi\colon  H \to \caH(\caE)$ a right fibration enjoying the following conditions.

\begin{enumerate}
\item $(\caE, \phi)$ is an additive $\infty$-category with genuine duality.
\item For every commutative square
\begin{equation*}
\begin{tikzcd}
A \ar[r, "f"] \ar[d, "\alpha"']         &B \ar[d, "\beta"] \\
	C \ar[r, "g"']                      &D
\end{tikzcd}
\end{equation*}
in $\caH(\caE)$ lying over a pushout square in $\caE$ such that $f$ lies over a cofibration, the induced square of spaces
\[
\begin{tikzcd}
H(D) \ar[r] \ar[d]          &H(B) \ar[d] \\
H(C) \ar[r]                 &H(A)
\end{tikzcd}
\]
is a pullback square.
\end{enumerate}
\end{definition}

Denote by $\Wald_\infty^\gd \subset \Wald_\infty^\d \times_{\Add^{hC_2}} \Add^\gd$ the full subcategory spanned by the small Waldhausen $\infty$-categories with genuine duality. Here $\Add^\gd$  denotes the $\infty$-category of additive $\infty$-categories with genuine duality, see \cite[Section 5.2]{realKthHSV}.

\begin{example}
Every stable $\infty$-category with genuine duality is a Waldhausen $\infty$-category with genuine duality, where every map is a cofibration. This example is explored in more detail in \cite[Section 6]{realKthHSV}.
\end{example}

Before moving towards the construction of the real $S_\bullet$-construction, we introduce the following notation that will be useful throughout the paper. 

\begin{notation}\label{not:tilde}Let $\caC$ be a $\Spc^{C_2}$-enriched $\infty$-category that admits cotensors with $C_2$, and $Z \in \caC$, then we denote the cotensor $Z^{C_2}$ by $\widetilde{Z \times Z}$. An explanation of this choice is found after \cite[Notation 3.9]{realKthHSV}.
\end{notation}

\begin{proposition} The $\infty$-category $\Wald_\infty^\gd$ is a real $\infty$-category and admits cotensors with $C_2$.
\end{proposition}

\begin{proof}
    This follows directly from \cite[Corollary 7.28]{realKthHSV}.
\end{proof}

\subsection{The real $S_\bullet$-construction}

This sections builds on the $S_\bullet$-construction defined by Barwick in \cite[Section 3]{barwick.wald} to include the genuine dualities presented beforehand. We originally introduced this in \cite[Section 8.2]{realKthHSV}, we refer the reader there for more details.

\begin{definition} Let $\caC$ be an exact $\infty$-category and $n \geq 0$.
We define  $$S^u(\caC)_n \subset \Fun(\Ar([n]),\caC)$$ to be the full subcategory
of functors $A\colon \Ar([n]) \to \caC$ such that for every $ 0 \leq i \leq n$
the image $A_{i,i}$ is zero and for every 
$0 \leq i \leq j \leq \ell \leq k \leq n$ the induced square

\begin{equation}\label{pullbu}
\begin{tikzcd}
A_{i,\ell} \ar[r]\ar[d] &  A_{i,k} \ar[d]  \\
A_{j,\ell} \ar[r] &  A_{j,k}
\end{tikzcd}
\end{equation}
in $\caC$ is an exact square.	
\end{definition}

The $\infty$-category $\Fun(\Ar([n]),\caC)$ is an exact $\infty$-category
with objectwise structure, however this exact structure is not inherited by $S(\caC)_n$.
In what follows we endow $S(\caC)_n$ with the structure of an exact $\infty$-category with fewer cofibrations and fibrations than what we would obtain from the objectwise one.

\begin{definition}\label{cofibrations_in_Se} Let $\caC$ be an exact $\infty$-category.
\begin{itemize}
\item We set a map $A \to B$ in $S(\caC)_n$ to be a cofibration if 
for every $1 \leq i < j \leq n $ the canonical maps
$ A_{0,1} \to B_{0,1} $ and
$ A_{0,j} \coprod_{A_{0,i}} B_{0,i} \to B_{0,j}$ are cofibrations.

\item We set a map $A \to B$ in $S(\caC)_n$ to be a fibration if for every $1 \leq i < j \leq n $ the canonical maps $A_{n-1,n} \to B_{n-1,n}$ and
$ A_{i,n} \to B_{i,n} \prod_{B_{j,n}} A_{j,n} $ are fibrations.
\end{itemize}
\end{definition}

\begin{proposition}[{\cite[Proposition 8.7]{realKthHSV}}]\label{exact_structure_on_S}
Let $\caC$ be an exact $\infty$-category and $n \geq 0$. The $\infty$-category $S(\caC)_n$ with the choice of cofibrations and fibrations as above forms an exact $\infty$-category and the embedding $S(\caC)_n \subset \Fun(\Ar([n]),\caC)$ is exact. 
\end{proposition}

\subsection*{The real $S_\bullet$-construction}

We prove in \cite[Corollary 7.18]{realKthHSV} that the real $\infty$-category $\Wald_\infty^\gd$ is cotensored over $\Cat_\infty^\gd$\textemdash the infty-category of $\infty$-categories with genuine dualities. In fact, for a small $\infty$-category $K$ with genuine duality and a small Waldhausen $\infty$-category $\caC$ with genuine duality, the cotensor $\caC^K$ has underlying $\infty$-category $\Fun(K, \caC)$ and objectwise fibrations and cofibrations. Using this, we endow the exact $\infty$-category $S(\caC)_n$ with a genuine duality as follows.

\begin{proposition}[{\cite[Proposition 8.8]{realKthHSV}}]\label{Se_has_gd}
Let $\caC$ be a Waldhausen  $\infty$-category with genuine duality and let $n \geq 0$.
The exact subcategory $S(\caC)_n \subset \Fun(\Ar([n]),\caC)$ is closed under the genuine duality of the functor $\infty$-category.
\end{proposition}

To define the real $S_\bullet$-construction we now want to assemble the categories $S(\caC)_n$ into a real simplicial object.

\begin{notation}
Let $\Ar\colon \underline{\Delta} \to \Cat_\infty^\gd$ the real functor 
$$\underline{\Delta} \subset \Cat_\infty^\gd \xrightarrow{\Hom_{ \Cat_\infty^\gd }([1], -)}  \Cat_\infty^\gd.$$

\end{notation}

\begin{corollary}

The real functor $\underline{\Delta}^\op \times \Wald_\infty^\gd  \to \Wald_\infty^\gd$ that sends $ (K,\caC) \mapsto \caC^{\Ar(K)}$ induces a real functor 
\[
\begin{tikzcd}[row sep=tiny]
S\colon \underline{\Delta}^\op \times \Wald_\infty^\gd \ar[r, "\sigma"]     &\Wald_\infty^\gd\\
\end{tikzcd}
\]
that sends $(K,\caC) \in \underline{\Delta}^\op \times \Wald_\infty^\gd$
to the genuine duality on $S(\caC)_n$ inherited from $\caC^{\Ar([n])}$
via Proposition \ref{Se_has_gd}.
\end{corollary}

\begin{definition}
The functor $\sigma$ as above is the transpose of a real functor $S_\bullet\colon \Wald_\infty^\gd \to \mathrm{rs}\Wald_\infty^\gd $, which we call the real Waldhausen $S_\bullet$-construction. From now on we will abuse notation by calling by $S$ both the real and usual $S_\bullet$-construction.
\end{definition}

An important feature of the real $S_\bullet$-construction is that, after taking its edgewise subdivision, we obtain a Segal object; see \cite[Proposition 8.15]{realKthHSV}. 

Before moving to recall the definition of the real $K$-theory space, we recall a fact about the real $S_\bullet$-construction that will be paramount to obtain the sought comparison.

\begin{notation}\label{notation_ambwald}
We denote by $\mathrm{Amb}(\caC) \subset \caC^{\widetilde{[1] \times [1]}}$ the full Waldhausen subcategory with genuine duality spanned by the exact squares.
\end{notation}

The embedding of $\infty$-categories with duality $\widetilde{[1] \times [1]} \subset \Ar([3])$ sending $(i,j) \mapsto (i,j+2)$ induces a map of Waldhausen $\infty$-categories with genuine duality as below $$\kappa\colon S(\caC)_3 \to \mathrm{Amb}(\caC) \subset {\caC^{\widetilde{[1] \times [1]}}}.$$

\begin{lemma}[{\cite[Corollary 8.14]{realKthHSV}}]\label{corqa}
Let $\caC$ be a Waldhausen $\infty$-category with genuine duality. The map
$$\kappa\colon S(\caC)_3 \to \mathrm{Amb}(\caC)$$ of Waldhausen $\infty$-categories with genuine duality is an equivalence.
\end{lemma}

\subsection{Real $K$-theory space}

In \cite[Section 8.2]{realKthHSV} we define a version of the $S_\bullet$-construction that plays nicely with the dualities involved in our objects. We then use this to define the real $K$-theory space for Waldhausen $\infty$-categories with genuine duality as follows.

\begin{definition}
Let $\caC$ be a Waldhausen $\infty$-category with genuine duality. We define the real $K$-theory space of $\caC$ as 
$$\KR(\caC)\coloneqq|  (-)^\simeq \circ S(\caC) \circ e|.$$
\end{definition}

\section{Setting the stage}

In this section we introduce the notion of $\infty$-category with genuine pro-duality, which generalizes the notion of $\infty$-category with pro-duality of \cite{Herm}. The main reason for us to consider pro-dualities is that the $\infty$-category of such is cotensored over the $\infty$-category of small $\infty$-categories, fact that is not true for genuine dualities, and that we use to define the hermitian $Q$-construction.

\subsection{$\infty$-categories with pro-duality}\label{produal}

Note that the non-trivial $C_2$-action on $\Cat_\infty$ (see \cite[Lemma 2.1]{realKthHSV}) together with the non trivial duality on $[1]$ induce a $C_2$-action on the over category $\Cat_{\infty / [1]}$. 

\begin{definition}\label{def:pro-duality}We define the $\infty$-category of $\infty$-categories with pro-dualities as the $\infty$-category $(\Cat_{\infty / [1]})^{hC_2}$, which we will denote simply by $\Cat_\infty^\mathrm{pd}$. We call an object of this category an $\infty$-category with pro-duality.
\end{definition}

\begin{lemma}\label{lem:duality_on_over_cat}
Let $\caC$ be an $\infty$-category with $C_2$-action and
$X \in \caC^{hC_2}.$ There is an induced $C_2$-action on $\caC_{/X}$
and a canonical equivalence  of $\infty$-categories
$$(\caC_{/X})^{hC_2} \simeq (\caC^{hC_2})_{/X}.$$
\end{lemma}

Using \cref{lem:duality_on_over_cat}, we obtain an alternative way of thinking of pro-dualities, as it directly follows that there is an equivalence of $\infty$-categories
$$(\Cat_{\infty / [1]})^{hC_2} \simeq (\Cat_{\infty}^{hC_2})_{/[1]}.$$

\begin{notation}
If $\caD \to[1]$ is an $\infty$-category with pro-duality,
we set $\caC\coloneqq \caD_0 \simeq \caD^\op_1$ 
and often write $\bar{\caC}$ for $\caD.$

We say that $\bar{\caC} \to [1]$ endows $\caC$ with a pro-duality. We call a map $\bar{\caC} \to \bar{\caD} $ in $(\Cat_{\infty / [1]})^{hC_2}$
a functor $\caC \to \caD$ preserving the pro-duality.
\end{notation}

\begin{example}\label{Ex_inclusion_dualities} Every duality is a pro-duality. Indeed, 
the $C_2$-equivariant adjunction
$$(-) \times [1]\colon \Cat_{\infty} \rightleftarrows \Cat_{\infty/[1]}\colon \Fun_{[1]}([1],-)$$
induces an adjunction $$\alpha\colon \Cat_{\infty}^{hC_2} \rightleftarrows (\Cat_{\infty / [1]})^{hC_2} \simeq (\Cat_{\infty}^{hC_2})_{/[1]}: \beta.$$
Since the functor $(-) \times [1]$ is an inclusion,
the functor $\alpha$ is an inclusion, too.
Via $\alpha$ we view $\infty$-categories with duality as $\infty$-categories with pro-duality.
\end{example}

\begin{remark} Every $\infty$-category with pro-duality has a canonical $\infty$-category with duality associated. We consider the functor 

$$(\Cat_{\infty}^{hC_2})_{/[1]}\simeq (\Cat_{\infty / [1]})^{hC_2} \to \Cat_{\infty}^{hC_2},$$
that maps a pro-duality to the total $\infty$-category with duality, $(\bar{\caC} \to [1]) \mapsto \bar{\caC}$.
\end{remark}

With this in mind, we make the following definition of hermitian objects.

\begin{definition}
Let $\overline{\caC}\to [1]$ be an $\infty$-category with pro-duality. We define its $\infty$-category of hermitian objects $\caH(\caC)$ as the $\infty$-category of hermitian objects of the total $\infty$-category with duality $\bar{\caC}$.
\end{definition}

\begin{remark}Let $\caC$ be an $\infty$-category with duality.
The $\infty$-category of hermitian objects of $\caC$
agrees with the $\infty$-category of hermitian objects of the underlying pro-duality of $\caC$. Indeed, since $\caH([1]) $ is contractible, there is a canonical equivalence
$$\caH(\caC) \simeq \caH(\caC)\times \caH([1]) \simeq \caH(\caC \times [1]).$$
\end{remark}

\begin{remark}\label{uujn}
Let $\bar{\caC} \to [1]$ be an $\infty$-category with pro-duality. Every hermitian object of $\caC$, i.e. of $\bar{\caC}$, lies over the unique hermitian object $0$ of $[1]$.
Thus $\caH(\caC) \simeq \caC \times_{\bar{\caC}} \caH(\caC).$
So there is a forgetful functor $\caH(\caC)\to \caC.$	
\end{remark}

\subsection{$\infty$-categories with genuine pro-duality}

\begin{definition}
An $\infty$-category with genuine pro-duality 
is a pair $(\bar{\caC} \to [1], \phi\colon H \to \caH(\caC))$,
where $\bar{\caC} \to [1]$ is an $\infty$-category with pro-duality and
$H \to \caH(\caC)$ is a right fibration.
\end{definition}

\begin{example}
Every $\infty$-category with duality $(\caC, \phi\colon H \to \caH(\caC))$ gives rise to an $\infty$-category with genuine pro-duality $(\caC\times[1] \to [1], \phi\colon H \to \caH(\caC))$.
\end{example}

\begin{definition}
The $\infty$-category of small $\infty$-categories with genuine pro-duality 
is the pullback $$ \Cat_{\infty}^{\mathrm{gp}}\coloneqq (\Cat_{\infty / [1]})^{hC_2} \times_{\Cat_{\infty}} \caR$$
of evaluation at the target
$\caR \to \Cat_{\infty}$ and the functor
$\caH\colon (\Cat_{\infty / [1]})^{hC_2} \to \Cat_{\infty}.$
\end{definition}

\begin{remark} There is a genuine version of what we observed after \cref{def:pro-duality}. Indeed, there is a canonical equivalence
$$\Cat_{\infty}^{\mathrm{gp}}=(\Cat_{\infty / [1]})^{hC_2} \times_{\Cat_{\infty}} \caR \simeq (\Cat_{\infty}^{hC_2})_{ / [1]} \times_{\Cat_{\infty}^{hC_2} } \Cat_{\infty}^\gd \simeq (\Cat_{\infty}^\gd)_{ / [1]},$$
where $[1]$ carries the standard genuine refinement.
\end{remark}

\begin{remark} We can see any $\infty$-category with genuine pro-duality as an $\infty$-category with genuine duality. Indeed, we have a localization $\Cat_{\infty}^{hC_2} \subset \Cat_\infty^\gd$
that allows us to view an $\infty$-category with duality as one with genuine duality, by considering the standard genuine refinement; see \cite[Remark 2.3.4]{realKthHSV}. This, in turn, gives rise to a localization 
\begin{equation}\label{locx}
(\Cat_{\infty / [1]})^{hC_2} \simeq (\Cat_{\infty}^{hC_2})_{ / [1]} \subset (\Cat_{\infty}^\gd)_{ / [1]}\simeq\Cat_{\infty}^{\mathrm{gp}} ,\end{equation}
where the left adjoint sends
$\bar{\caC} \to [1] $ to $(\bar{\caC} \to [1], \id: \caH(\caC) \to \caH(\caC))$
and the right adjoint forgets the genuine refinement.
\end{remark}

\begin{remark}\label{rmk:inclusion_gd_to_gp} 
\begin{enumerate}
\item We can also view $\infty$-categories with genuine duality as $\infty$-categories with genuine pro-duality. The inclusion $(-)\times [1]\colon \Cat_{\infty}^{hC_2} \hookrightarrow (\Cat_{\infty / [1]})^{hC_2}$ gives rise to an inclusion
$$\Cat_\infty^\gd = \Cat_{\infty}^{hC_2} \times_{\Cat_{\infty}} \caR \hookrightarrow (\Cat_{\infty / [1]})^{hC_2} \times_{\Cat_{\infty}} \caR=\Cat_{\infty}^{\mathrm{gp}},$$
that takes pullback along the map
$[1] \to [0]$ of $\infty$-categories with standard genuine duality.

\item This inclusion admits a right adjoint, which is the functor
$$ \beta \times_{\Cat_{\infty}} \caR:(\Cat_{\infty / [1]})^{hC_2} \times_{\Cat_{\infty}} \caR \to \Cat_{\infty}^{hC_2} \times_{\Cat_{\infty}} \caR,$$
where $\beta$ is the right adjoint of the inclusion
$ \Cat_{\infty}^{hC_2} \subset (\Cat_{\infty / [1]})^{hC_2}$
of Example \ref{Ex_inclusion_dualities}.
\end{enumerate}
\end{remark}

In the following diagram we summarize the main definitions of dualities prevoiusly introduced and their relation. 

\begin{tz}
\node[](1)  {$\Cat_\infty^{hC_2}$};
\node[above of=1,yshift=0.15cm](1') {$\Cat^d_\infty$};

\node[la] at ($(1'.south)!0.5!(1.north)$) {$\rotatebox{270}{$\coloneqq$}$}; 

\node[right of=1,xshift=2.5cm](2) {$\Cat_\infty^\mathrm{gd}$};

\node[below of=1,yshift=-1cm](3) {$\slice{(\Cat_\infty^{hC_2})}{[1]}$};
\node[below of=3](3') {$\left(\slice{\Cat_\infty}{[1]}\right) ^{hC_2}$};
\node[below of=3'](3'') {$\Cat_\infty^{\mathrm{pd}}$};

\node[la] at ($(3.south)!0.6!(3'.north)$) {$\rotatebox{270}{$\simeq$}$}; 
\node[la] at ($(3'.south)!0.5!(3''.north)$) {$\rotatebox{90}{$\coloneqq$}$}; 

\node[right of=3,xshift=2.5cm](4) {$\slice{(\Cat_\infty^{\mathrm{gd}})}{[1]}$};
\node[below of=4,](4') {$\Cat_\infty^{\mathrm{gp}}$};
\node[below of=4'](4'') {$\Cat_\infty^{\mathrm{pd}}\times\caR$};

\node[la] at ($(4.south)!0.5!(4'.north)$) {$\rotatebox{270}{$\simeq$}$};
\node[la] at ($(4'.south)!0.5!(4''.north)$) {$\rotatebox{270}{$\coloneqq$}$}; 
 
\draw[right hook->] (1) to (2);
\draw[right hook->] (3) to (4);
\draw[right hook->] (1) to (3);
\draw[right hook->] (2) to (4);
\end{tz}

\subsection{Internal homs for genuine pro-dualities} In this sub section we present results of the internal homs for $\infty$-categories with (genuine) pro-duality that will be necessary for the hermitian $Q$-construction.

\begin{remark}[Internal homs in $\Cat_{\infty / [1]}$]\label{rmk:internalhom_Cat_over_1}

By \cite[Proposition B.3.14]{lurie.higheralgebra} the $\infty$-category $\Cat_{\infty / [1]}$ is cartesian closed. For functors $\caA \to [1],\caB \to [1]$ we write 
$\Fun^{[1]}(\caA,\caB)\to[1]$ for the internal hom.

As proven in \cite[Remark 2.36]{monadicity}, these internal homs are better understood via the following canonical equivalences. Let $\caA,\caB$ be two $\infty$-categories over $[1]$. There are canonical equivalences:
\begin{enumerate}
    \item for $i\in [1]$, $$\{i\}\times_{[1]}\Fun^{[1]}(\caA,\caB) \simeq \Fun(\caA_i,\caB_i)$$
    \item $$\Fun^{[1]}(\caC \times [1],\caD \times [1]) \simeq \Fun(\caC,\caD)\times [1].$$
\end{enumerate}
\end{remark}

\begin{remark}[Internal homs for pro-dualities] Now, using \cref{rmk:internalhom_Cat_over_1} together with \cite[Remark 2.6]{realKthHSV}, we know that also $(\Cat_{\infty / [1]})^{hC_2}$ is cartesian closed and moreover the forgetful functor $(\Cat_{\infty / [1]})^{hC_2} \to \Cat_{\infty / [1]}$ preserves the internal hom.
For $\infty$-categories with pro-duality $\bar{\caC}\to [1], \bar{\caD}\to [1]$
we write $\overline{\Fun(\caC,\caD)} $ for $\Fun^{[1]}(\bar{\caC},\bar{\caD}).$
\end{remark}

In the next lemma we state the existence of two functors relating the hermitian objects of the internal hom of two categories with pro-duality to the hermitian objects of the categories with pro-duality. The first of these functors is also used to describe the internal hom of the category of small categories with genuine pro-duality.

\begin{lemma} Let $\bar{\caC}\to [1], \bar{\caD}\to [1]$ be $\infty$-categories with pro-duality. Then we have two canonical functors as below:

$$ \caH(\Fun(\caC,\caD)) \to \Fun(\caH(\caC),\caH(\caD)),$$

$$\caH(\Fun_{[1]}(\bar{\caC},\bar{\caD})) \to \caH(\Fun(\caC,\caD)).$$

\end{lemma}

\begin{proof}

There is a canonical map
$$ \Fun^{[1]}(\bar{\caC},\bar{\caD}) \times_{[1]}\bar{\caC} \to \bar{\caD}$$
in $(\Cat_{\infty / [1]})^{hC_2}$
that induces a functor
$$ \caH(\Fun(\caC,\caD)) \times \caH(\caC)=\caH(\Fun^{[1]}(\bar{\caC},\bar{\caD})) \times \caH(\bar{\caC}) \simeq \caH(\Fun^{[1]}(\bar{\caC},\bar{\caD})) \times_{\caH([1])} \caH(\bar{\caC}) $$$$ \simeq \caH(\Fun^{[1]}(\bar{\caC},\bar{\caD}) \times_{[1]}\bar{\caC}) \to \caH(\bar{\caD}) =\caH(\caD)$$
corresponding to a functor $ \caH(\Fun(\caC,\caD)) \to \Fun(\caH(\caC),\caH(\caD))$.
There is a canonical map
$$\Fun_{[1]}(\bar{\caC},\bar{\caD}) \times [1]\to \Fun^{[1]}(\bar{\caC},\bar{\caD})$$
in $(\Cat_{\infty / [1]})^{hC_2}$
that induces a functor
$$\caH(\Fun_{[1]}(\bar{\caC},\bar{\caD})) \simeq  \caH(\Fun_{[1]}(\bar{\caC},\bar{\caD})) \times \caH([1]) \simeq$$$$ \caH(\Fun_{[1]}(\bar{\caC},\bar{\caD}) \times [1]) \to \caH(\Fun^{[1]}(\bar{\caC},\bar{\caD}))= \caH(\Fun(\caC,\caD)).$$
\end{proof}

We now move towards studying the internal hom for $\infty$-categories with genuine pro-duality. Again, using \cite[Proposition 2.43]{realKthHSV} we know that the $\infty$-category $\Cat^\mathrm{gp}$ is cartesian closed. To give a description of the internal hom, we use the pullback square in the statement of \cite[Proposition 2.43]{realKthHSV}, which specialized to this case gives us this result.

\begin{proposition}[Internal homs for genuine pro-duality]\label{prop:internal_hom_produality}
The $\infty$-category $\Cat_{\infty}^{\mathrm{gp}}$ of $\infty$-categories with genuine pro-duality is cartesian closed. Moreover, for any two $\infty$-categories with genuine pro-duality $(\bar{\caC}\to [1], \phi\colon H\to \caH(\caC))$ and $(\bar{\caD}\to [1],\psi\colon T\to\caH(\caD))$, the internal hom $(\Fun(\caC,\caD),H_{\Fun(\caC,\caD)}\to \caH(\Fun(\caC,\caD)))$
is the $\infty$-category with genuine pro-duality
 $$(\overline{\Fun(\caC,\caD)}\to [1], \caH(\Fun(\caC,\caD)) \times_{ \Fun(H,\caH(\caD)) } \Fun(H,T) \to  \caH(\Fun(\caC,\caD)),$$
where the map in the second coordinate is just the projection.
\end{proposition}

\begin{corollary}
The restriction of the
canonical left action of $\Cat_\infty^\mathrm{gp}$ on itself (via the product) along the inclusion $\Cat_\infty^\gd \subset (\Cat_\infty^\gd)_{ / [1]}\simeq \Cat_\infty^\mathrm{gp}$ is a closed left action of $\Cat_\infty^\gd$ on $\Cat_\infty^\mathrm{gp}$.   
Moreover, for any two $\infty$-categories with genuine pro-duality $(\bar{\caC}\to [1], \phi\colon H\to \caH(\caC))$ and $(\bar{\caD}\to [1],\psi\colon T\to\caH(\caD))$, the external hom is the $\infty$-category with genuine duality 
$$(\Fun_{[1]}(\bar{\caC},\bar{\caD}), \caH(\Fun_{[1]}(\bar{\caC},\bar{\caD})) \times_{ \Fun(H,\caH(\caD)) } \Fun(H,T) \to \caH(\Fun_{[1]}(\bar{\caC},\bar{\caD}))),$$
where the map in the second coordinate is just the projection.
\end{corollary}

\begin{proof}
This follows directly from \cref{prop:internal_hom_produality} which shows that the $\infty$-category $\Cat_\infty^\mathrm{gp}$ is cartesian closed, and the fact that the inclusion $\Cat_\infty^\mathrm{gd}\to\Cat_\infty^\mathrm{gp}$ admits a right adjoint, see \cref{rmk:inclusion_gd_to_gp}. 
\end{proof}

\begin{remark}
Both adjoints in the localization \cref{locx}, as well the inclusion
$\Cat_\infty^\gd \hookrightarrow (\Cat_\infty^\gd)_{ / [1]}$,
preserve the internal homs. 
\end{remark}

\subsection{A different model of $\Cat_\infty^{\mathrm{pd}}$ and $\Cat_\infty^{\mathrm{gp}}$}

In the following we construct a different model for $\infty$-categories 
with pro-duality and genuine pro-duality.
We use this model to give another description of the internal hom
in $\Cat_\infty^{\mathrm{pd}}$ and $\Cat_\infty^{\mathrm{gp}}$,
which we use to analyze the hermitian $Q$-construction presented in \cref{hermq}.

\begin{notation}
We now consider the pullback $\Sym\coloneqq\Cat_{\infty} \times_{ \Cat_{\infty}[C_2]} \caR[C_2]$; see \cite{Herm}. 
\end{notation}

An object of $\Sym$ is a pair $(\caC, \phi\colon \caV \to \widetilde{\caC \times \caC})$,
where $\phi$ is a $C_2$-equivariant right fibration.
The $C_2$-equivariant right fibration $\phi$
classifies a $C_2$-equivariant functor $\widetilde{\caC^\op \times \caC^\op} \to \caS,$ a symmetric functor in the sense of \cite[Definition 5.9]{hls}. Thus, an object of $\Sym$ corresponds to an $\infty$-category equipped with a symmetric functor.

\begin{proposition}\label{prop:equiv_pd_sym}There is a canonical equivalence
$$\Cat_\infty^\mathrm{pd} \simeq \Sym$$
over $\Cat_\infty$ that sends an $\infty$-category with pro-duality $\bar{\caC} \to [1] $ to the pair $$(\caC, \widetilde{\caC \times \caC} \times_{\widetilde{\overline{\caC} \times \overline{\caC}}} \Tw(\bar{\caC}) \longrightarrow \widetilde{\caC \times \caC}).$$    
\end{proposition}

\begin{proof}
Recall that, by definition, $\infty$-categories with pro-duality are $\Cat_\infty^\mathrm{pd}=(\Cat_{\infty / [1]})^{hC_2}$. By \cite[Proposition 5.23]{Herm} we know that there is a canonical $C_2$-equivariant equivalence
$$\Psi\colon \Cat_{\infty / [1]} \to (\Cat_{\infty} \times \Cat_{\infty}) \times_{ \Cat_{\infty}} \caR$$ that sends a functor $ \caM \to [1]$
to the tuple $$(\caM_0, \caM_1, (\caM_0 \times \caM_1^\op) \times_{(\caM \times \caM^\op)} \Tw(\caM) \to \caM_0 \times \caM_1^\op),$$
where $\Cat_{\infty} $ and $\caR$ on the right hand side carry the trivial action.

Thus $\Psi$ induces on homotopy $C_2$-fixed points an equivalence
$$(\Cat_{\infty / [1]})^{hC_2} \simeq \Cat_{\infty} \times_{ \Cat_{\infty}[C_2]} \caR[C_2]$$
that sends an $\infty$-category with pro-duality $\bar{\caC} \to [1] $ to the pair $$(\caC, \widetilde{\caC \times \caC} \times_{\widetilde{\overline{\caC} \times \overline{\caC}}} \Tw(\bar{\caC}) \to \widetilde{\caC \times \caC}).$$
\end{proof}

\begin{notation}
We consider the pullback $$\mathrm{Qu}\coloneqq \Sym \times_{\Cat_{\infty}} \caR$$
 along evaluation at the target and the functor taking homotopy $C_2$-fixed points of the source of a $C_2$-equivariant right fibration. 
\end{notation}

\begin{corollary}\label{equiv_pg_Qu}
There is a canonical equivalence
$$\Cat_{\infty}^\mathrm{gp}= \Cat_{\infty}^\mathrm{pd} \times_{\Cat_{\infty}} \caR \simeq \mathrm{Qu}$$
that sends $(\bar{\caC} \to [1], \phi\colon H \to \caH(\caC))$
to $(\caC, \caV \to \widetilde{\caC \times \caC}, \phi\colon H \to \caV^{hC_2})$.
\end{corollary}

\cref{prop:equiv_pd_sym} allows us to obtain a description of the hermitian objects of an $\infty$-category with pro-duality as follows.

\begin{lemma}\label{uujnhn}
Let $\bar{\caC} \to [1]$ be an $\infty$-category with pro-duality
and $\caV \to \widetilde{\caC \times \caC}$ the associated $C_2$-equivariant right fibration.
There is a canonical equivalence
$$ \caH(\caC) \simeq \caV^{hC_2}.$$		
\end{lemma}

\begin{proof}
We consider the following chain of equivalences.
\begin{align*}
     \caH(\caC)     &\coloneqq\caH(\bar{\caC})\\
                    &\simeq \caC \times_{\bar{\caC}} \caH(\bar{\caC})\\
                    & \simeq (\widetilde{\caC \times \caC} \times_{ \widetilde{\overline{\caC} \times \overline{\caC}}} \Tw(\bar{\caC}))^{hC_2}\\
                    &\simeq \caV^{hC_2}.
\end{align*}
The first equivalence holds because every hermitian object of $\bar{\caC}$ lies over an object of $\caC \subset \bar{\caC}$, for the second equivalence we use that taking homotopy $C_2$-fixed points preserves limits and since $\caH(\bar{\caC})= \Tw(\bar{\caC})^{hC_2}$, finally the third equivalence follows from the construction of the equivalence of \cref{prop:equiv_pd_sym}.
\end{proof}

\subsection*{Internal homs in $\Sym$ and $\mathrm{Qu}$} These descriptions, although technical, will aid in the comparison we are after.

Next we consider a variation of the $\infty$-category $\mathrm{Qu}$, that is slightly more general.

\begin{notation}
Let $$\Xi\coloneqq \Cat_{\infty} \times_{\Cat_{\infty}[C_2]} \Fun([1],\Cat_{\infty})[C_2] \times_{\Cat_{\infty}} \Fun([1],\Cat_{\infty})$$ be the pullbacks of the functors $$\Cat_{\infty} \to  \Cat_{\infty}[C_2]\hspace{1em}\text{ that sends }\hspace{1em}\caC \mapsto \widetilde{\caC \times \caC},$$
and evaluation at the target 
$$\ev_1\colon \Fun([1],\Cat_{\infty})[C_2] \to \Cat_{\infty}[C_2]$$
for the first pullback, and the following functors for the second pullback
$$\Fun([1],\Cat_{\infty})[C_2] \xrightarrow{\ev_0} \Cat_{\infty}[C_2] \xrightarrow{(-)^{hC_2}} \Cat_{\infty}, $$
which evaluates at the source, and evaluation at the target
$$\ev_1\colon \Fun([1],\Cat_{\infty}) \to \Cat_{\infty}.$$
\end{notation}

\begin{remark}
The $\infty$-category $\Xi$ admits small limits as a pullback of $\infty$-categories with small limits along functors preserving small limits.  
\end{remark}

In \cref{lemwa} we show that $\infty$-category $\Xi$ is cartesian closed. In order to state it, we introduce the following notation, that 

\begin{notation}Let $(\caB, \alpha, \rho), (\caC, \beta, \phi), (\caD, \gamma, \psi) \in \Xi$.
Let $$ \mathrm{Mor}_\Xi((\caC, \beta,\phi), (\caD, \gamma,\psi))\coloneqq $$$$\Fun(\caC,\caD) \times_{\Fun(\caV,\widetilde{\caD \times \caD})^{hC_2}} \Fun(\caV, \caW)^{hC_2} \times_{\Fun(H, \caW^{hC_2})} \Fun(H,H').$$
\end{notation}

\begin{remark}\label{rewqy}
The maximal subspace $ \mathrm{Mor}_\Xi((\caC, \beta,\phi), (\caD, \gamma,\psi))^\simeq$ is the mapping space $\map_\Xi((\caC, \beta,\phi), (\caD, \gamma,\psi)).$
\end{remark}

\begin{lemma}\label{lemwa}
Let $(\caB, \alpha, \rho), (\caC, \beta, \phi), (\caD, \gamma, \psi) \in \Xi$. There is a canonical equivalence
$$ \map_\Xi((\caB, \alpha, \rho) \times (\caC, \beta,\phi), (\caD, \gamma,\psi))
\simeq \map_\Xi((\caB, \alpha, \rho) , (\caD, \gamma,\psi)^{(\caC, \beta,\phi)}),$$
where the internal hom $(\caD, \gamma,\psi)^{(\caC, \beta,\phi)}\in\Xi$ is the triple $(\Fun(\caC,\caD), \lambda, \sigma)$ where $\lambda$ and $\sigma$ are the projections
$$\lambda\colon \widetilde{\Fun(\caC,\caD) \times \Fun(\caC,\caD)} \times_{\Fun(\caV,\widetilde{\caD \times \caD})} \Fun(\caV, \caW) \longrightarrow \widetilde{\Fun(\caC,\caD) \times \Fun(\caC,\caD)},\text{ and}$$
$$\sigma\colon \mathrm{Mor}_\Xi((\caC, \beta,\phi), (\caD, \gamma,\psi)) \to\Fun(\caC,\caD) \times_{\Fun(\caV,\widetilde{\caD \times \caD})^{hC_2}} \Fun(\caV, \caW)^{hC_2}.$$
\end{lemma}

\begin{proof}
The equivalence follows immediately from the definitions.	
\end{proof}

Since we have just proved that the $\infty$-category $\Xi$ is cartesian closed in \cref{lemwa}, we know that  $\mathrm{Qu} \subset \Xi$ is cartesian closed as well.
Under the equivalence in \cref{equiv_pg_Qu} the internal homs correspond, which gives us the following corollary.

\begin{corollary}
Let $(\caC, H_\caC), (\caD, H_\caC)$ be 
$\infty$-categories with genuine pro-duality
corresponding to $(\caC, \beta,\phi), (\caD, \gamma,\psi)\in \Xi.$ There is a canonical equivalence
$$ H_{\Fun(\caC,\caD)} 
\simeq \mathrm{Mor}_\Xi((\caC, \beta,\phi), (\caD, \gamma,\psi)).$$

\end{corollary}

\begin{proof}
From the definitions, it follows that there is an equivalence
$$H_\caC \simeq \mathrm{Mor}_\Xi(\ast, (\caD, \gamma,\psi)),$$
where $\ast$ is the final object of $\Xi$, and so a canonical equivalence 
$$ H_{\Fun(\caC,\caD)} 
\simeq \mathrm{Mor}_\Xi(*,(\caD, \gamma,\psi)^{(\caC, \beta, \phi)}) \simeq \mathrm{Mor}_\Xi((\caC, \beta,\phi), (\caD, \gamma,\psi)). $$	
\end{proof}

\subsection*{Cotensors on $\mathrm{Qu}$}

\begin{lemma}\label{adj_Cat_Xi} There is an adjunction

\begin{tz}
\node[](1)    {$\Cat_\infty$};
\node[right of=1,xshift=1.5cm](2)    {$\Xi$};
\node[la] at ($(1.east)!0.5!(2.west)$) {$\bot$}; 
\draw[->] ($(1.east)+(0,5pt)$) to node[above,la]{$\rho$} ($(2.west)+(0,5pt)$);
\draw[->] ($(2.west)-(0,5pt)$) to node[below,la]{} ($(1.east)-(0,5pt)$);

\end{tz}
where the functor $\rho\colon\Cat_\infty\to \xi$ is given by  $K \mapsto (K, K \to \widetilde{K \times K}, K \to K^{BC_2})$ using the diagonal morphisms, and
the right adjoint sends $(\caC, \beta\colon \caV \to \widetilde{\caC \times \caC},\phi\colon H \to \caV^{hC_2})$ to $H.$
Moreover, the functor $\rho$ is fully faithful.
\end{lemma}

\begin{proof}
There is a canonical equivalence
$$\mathrm{Mor}_\Xi(\rho(K), (\caC, \beta,\phi))	\simeq $$$$ \Fun(K, \caC) \times_{\Fun(K,\widetilde{\caC \times \caC})^{hC_2}} \Fun(K, \caV)^{hC_2} \times_{\Fun(K, \caV^{hC_2})} \Fun(K,H) \simeq \Fun(K,H), $$
where the last equivalence holds since the first and third functors in the pullback are equivalences.
\end{proof}

Restricting the action of $\Xi$ on itself, that comes from the cartesian product, along $\rho\colon \Cat_\infty \to \Xi$ 
gives an action of $\Cat_\infty$ on $\Xi$,
which is closed because the action of $\Xi$ on itself is closed by Lemma \ref{lemwa} and $\rho$ admits a right adjoint by \cref{adj_Cat_Xi}.
We show the reader of how the external hom looks like in this particular case.

\begin{remark}\label{coten} Let $(\caC,\beta,\phi)$ be in $\Xi$ with $\beta\colon\caV\to\widetilde{\caC\times\caC}$ and $\phi\colon H\to \caV^{hC_2}$, and let $K$ be an $\infty$-category. Then the cotensor $(\caC, \beta,\phi)^K$ is the triple given by the following components:
\begin{itemize}
    \item[-]  $((\caC, \beta,\phi)^K)_1=\Fun(K,\caC)$
    \item[-] $((\caC, \beta,\phi)^K)_2\colon \Fun(K, \caV) \xrightarrow{\Fun(K, \beta)} \Fun(K,  \widetilde{\caC \times \caC}) \simeq \widetilde{ \Fun(K,\caC) \times  \Fun(K,\caC)}$
    \item[-] $((\caC, \beta,\phi)^K)_3\colon \Fun(K, H) \xrightarrow{\Fun(K,\phi)} \Fun(K, \caV^{hC_2}) \simeq \Fun(K, \caV)^{hC_2}$
\end{itemize}
Moreover, if $(\caC,\beta,\phi)$ is in $\mathrm{Qu}$, then  $(\caC, \beta,\phi)^K$ is in $\mathrm{Qu}$ as well.
\end{remark}

\begin{notation}\label{notation_coten}
Let $K$ be an $\infty$-category and $(\caC,\caH_\caC)$ 
an $\infty$-category with genuine pro-duality corresponding to
$(\caC,  \beta, \phi) \in \mathrm{Qu}$ via the equivalence of \cref{equiv_pg_Qu}.
We will denote by $(\caC,\caH_\caC)^K$
the $\infty$-category with genuine pro-duality corresponding to $(\caC,  \beta,\phi)^K.$
\end{notation}

\begin{lemma}\label{cotensor_butnotreally}
Let $(\caC,H_\caC),(\caD,H_\caD)$ be $\infty$-categories with genuine pro-duality
and $K$ an $\infty$-category.
There is a canonical equivalence
$$ \map_{\Cat_\infty}(K,H_{\Fun(\caC,\caD)}) \simeq \map_{\Cat_{\infty}^\mathrm{gp}}((\caC,H_\caC),(\caD,H_\caD)^K). $$
\end{lemma}

\begin{proof}Consider $(\caC,\beta, \phi),(\caD,\gamma, \psi)$ in $\mathrm{Qu}$ the corresponding objects to $(\caC,H_\caC),(\caD,H_\caD)$ respectively, via \cref{equiv_pg_Qu}.
Then we have the chain of equivalences below:
$$ \map_{\Cat_\infty}(K,H_{\Fun(\caC,\caD)}) \simeq \map_{\mathrm{Qu}}((\caC,\beta, \phi),(\caD,\gamma, \psi)^K) \simeq \map_{\Cat_{\infty}^\mathrm{gp}}((\caC,H_\caC),(\caD,H_\caD)^K). $$
\end{proof}

The next result will come in hand later. 

\begin{proposition}\label{maaap}
Let $(\caC,H_\caC), (\caD,H_\caD)$ be $\infty$-categories with genuine pro-duality. There is a canonical map
$$ (\Fun(\caC,\caD),H_{\Fun(\caC,\caD)}) \to (\caD,H_\caD)^{H_\caC}$$
of $\infty$-categories with genuine pro-duality
that induces the following functors:
\begin{enumerate}
    \item the canonical functor $\Fun(\caC,\caD) \to \Fun(H_\caC,\caD)$ on underlying $\infty$-categories,
    \item the functor $\caH(\Fun(\caC,\caD)) \to \Fun(\caH(\caC),\caH(\caD))\to \Fun(H_\caC,\caH(\caD))$ on hermitian objects, and 
    \item the functor $H_{\Fun(\caC,\caD)} \to \Fun(H_\caC,H_\caD)$ on genuine refinements.
\end{enumerate}
\end{proposition}

\begin{proof} To construct the map in the statement, let
 $(\caB,H_\caB), (\caC,H_\caC), (\caD,H_\caD),$ be $\infty$-categories with genuine pro-duality and let us consider a map
$$\alpha\colon (\caB,H_\caB) \to (\Fun(\caC,\caD),H_{\Fun(\caC,\caD)})$$ of $\infty$-categories with genuine pro-duality. 
We have an associated map $$(\caC,H_\caC) \to (\Fun(\caB,\caD),H_{\Fun(\caB,\caD)})$$ of $\infty$-categories with genuine pro-duality that in turns gives rise to a functor $H_\caC \to H_{\Fun(\caB,\caD)}$
corresponding to a map $ (\caB,H_\caB) \to (\caD,H_\caD)^{H_\caC}$
of $\infty$-categories with genuine pro-duality.
For $\alpha$ the identity we obtain the desired map
$$ (\Fun(\caC,\caD),H_{\Fun(\caC,\caD)}) \to (\caD,H_\caD)^{H_\caC}$$
of $\infty$-categories with genuine pro-duality.	
\end{proof}

\section{Hermitian \texorpdfstring{$Q$}\ -construction}\label{hermq}

The aim of this section is to define a hermitian $Q$-construction that both extends that of \cite{Calmes_etal2} and facilitates the comparison with the existing real $S_\bullet$-construction of \cite{realKthHSV}, and in doing that serves as a pivot for the comparison of real $K$-theory spaces of \cref{sec:comparison_Kth}.

\begin{definition}\label{def:Qconstruction}

Let $\caD$ be an $\infty$-category, $\caC, \caF \subset \caD$ be wide subcategories
and $n \geq 0$. 
The $\caQ$-construction is the full subfunctor $\caQ(\caD,\caC,\caF)$
of the functor $$\Delta^\op \to \Cat_{\infty} \hspace{0.8em}\text{ that maps }\hspace{0.8em} [n] \mapsto \Fun(\Tw([n]), \caD)$$
such that for every $n \geq0$ the full subcategory 
$$\caQ(\caD,\caC,\caF)_{[n]} \subset \Fun(\Tw([n]), \caD)$$
consists of the functors $X\colon \Tw([n]) \to \caD$ such that for any 
$ 0 \leq i \leq j \leq \ell \leq k \leq n$
the square
\begin{equation*}\label{pushhh}
\xymatrix{X_{i,k}\ar[r] \ar[d] & X_{j,k} \ar[d] \\
X_{i,\ell} \ar[r] & X_{j,\ell}}
\end{equation*}
is a pullback square,
for any $ 0 \leq i \leq \ell \leq k \leq n$
the map $X_{i,k} \to X_{i,\ell}$ is a fibration and 
for any $ 0 \leq i \leq j \leq \ell \leq n$
the map $X_{i,\ell} \to X_{j,\ell}$ is a cofibration.
\end{definition}

\begin{remark}
The embedding $\caQ(\caD,\caC,\caF)_{[0]} \subset \Fun(\Tw([0]),\caD) \simeq \caD$ is an equivalence.	
\end{remark}

The next result gives a sufficient condition for the $Q$-construction to be Segal. 

\begin{proposition}\label{lemas}
Let $(\caD, \caC,\caF)$ be an exact $\infty$-category.
The simplicial object given by the $Q$-construction
$\caQ(\caD, \caC,\caF)\colon \Delta^\op \to \Cat_\infty$ is Segal.
\end{proposition}

\begin{proof}

Let us consider the category $\caP(\caD)$ of presheaves on $\caD$, and denote $\caQ(\caP(\caD))$ to the Q-construction $\caQ(\caP(\caD)), \caP(\caD), \caP(\caD))$. 
For $n \geq 2$, we know that the Segal map
$$\caQ(\caP(\caD))_{[n]} \to \caQ(\caP(\caD))_{[1]} \times_{\caP(\caD)} \dots \times_{\caP(\caD)} \caQ(\caP(\caD))_{[1]}$$ is an equivalence (see \cite[Proposition 3.2]{Haug_Segal} or \cite[Proposition 5.14]{Haug_spans}). By the Yoneda-lemma this functor restricts to a fully faithful functor
$$\caQ(\caD, \caC,\caF)_{[n]} \to \caQ(\caD,\caC,\caF)_{[1]} \times_{\caD} \dots \times_{\caD} \caQ(\caD,\caC,\caF)_{[1]}.$$
Moreover, such restriction is essentially surjective\textemdash and therefore and equivalence\textemdash since the hypothesis on $(\caD,\caF,\caC)$ guarantees that every object of $\caQ(\caP(\caD))_{[n]}$ belongs to $\caQ(\caD, \caC,\caF)_{[n]}$ if its image under the Segal map belongs to 
$$ \caQ(\caD,\caC,\caF)_{[1]} \times_{\caD} \dots \times_{\caD} \caQ(\caD,\caC,\caF)_{[1]}$$
using the hypothesis.
\end{proof}

In the next definition we give candidates for cofibrations and fibrations for every level of $\caQ(\caD)$. 

\begin{definition}\label{definition_exact} Let $\caD$ be an exact $\infty$-category. 
\begin{itemize}
\item For $n=0$, we use the equivalence $\caQ(\caD)_{[0]} \simeq \caD$ and choose the (co)fibrations to be those of $\caD$.
\item For $n=1$, we say that a map $(X \to Y, X \to Z) \to (X' \to Y', X' \to Z')$ in $\caQ(\caD)_{[1]}$ is
\begin{itemize}
    \item[-] a cofibration if the maps $ 
X \to X', Z \coprod_X {X'} \to Z', Y \coprod_X {X'} \to Y'$ are cofibrations, and
    \item[-] a fibration if the maps $ 
Y \to Y', Z \to Z', X \to Z \prod_{Z'} {X'}, X \to Y \prod_{Y'} {X'}$ are fibrations.
\end{itemize}
\item For $n \geq 2$, we say that a map in $\caQ(\caD)_{[n]} \simeq\caQ(\caD)_{[1]} \times_{\caD} \dots \times_{\caD} \caQ(\caD)_{[1]}$ is a (co)fibration if the 
images under the projections are (co)fibrations.
\end{itemize}
\end{definition}

\begin{remark}\label{Q1_Amb_underlying_equiv}
Let $\caD $ be an exact $\infty$-category. The canonical functor $\Fun([1]\times[1],\caD) \to \Fun(\Lambda^2_0,\caD)$ restricting along the canonical embedding $\Lambda^2_0 \subset [1]\times[1]$ 
restricts to an equivalence
$\mathrm{Amb}(\caD) \to \caQ(\caD)_{[1]}$
of exact $\infty$-categories.
\end{remark}

\begin{proposition}\label{corx}
Let $\caD$ be an exact $\infty$-category and $n \geq 0$.
The $\infty$-category $\caQ(\caD)_{[n]}$ is exact with the choice of cofibrations and fibrations from above.
	
\end{proposition}

\begin{proof} By definition we know it for $n=0$. For $n=1$, note that  \cref{corqa} and \cref{Q1_Amb_underlying_equiv} gives us, respectively, the following equivalences
$$ S(\caD)_3 \simeq \mathrm{Amb}(\caD)\hspace{1em}\text{and}\hspace{1em}\mathrm{Amb}(\caD) \simeq \caQ(\caD)_{[1]}.$$
Since we know that $S(\caD)_3$ is an exact $\infty$-category by \cref{exact_structure_on_S}, and that cofibrations and fibrations correspond
under the equivalence $S(\caD)_3 \simeq \caQ(\caD)_{[1]}$, it follows that $\caQ(\caD)_{[1]}$ is an exact category. Finally, for $n\geq 2$ the statement follows from the definition of (co)fibrations and the fact that pushout and pullbacks are computed component-wise in pullbacks of $\infty$-categories.
\end{proof}

\begin{remark}\label{rmk:Q_exact}
By \cref{corx}, given an exact $\infty$-category $\caD$, the Segal object $\caQ(\caD) \colon \Delta^\op \to \Cat_\infty$ canonically lifts to a Segal object $\Delta^\op \to \Exact_\infty$.
\end{remark}

\subsection{Hermitian $Q$-construction}
We now enhance the $Q$-construction from a Segal object in $\Cat_\infty$ to a Segal object in $\Wald_\infty^\gd.$ 

\begin{definition}
A Waldhausen $\infty$-category with genuine pro-duality is a triple $$((\caD,\phi), \caC,\caF),$$
where $(\caD,\phi)$ is an $\infty$-category with genuine pro-duality
and $(\caD, \caC,\caF)$ is an exact $\infty$-category.
\end{definition}

\begin{definition} We define $\infty$-category of small Waldhausen $\infty$-categories with genuine pro-duality as the pullback
$$\Wald_\infty^{\mathrm{gp}} \coloneqq \Cat_\infty^{\mathrm{gp}} \times_{\Cat_\infty} \Exact_\infty $$ of the forgetful functor
$\Cat_\infty^{\mathrm{gp}} \to \Cat_\infty$ that maps $(\caC, \phi) \mapsto \caC$ along the functor $ \Exact_\infty \to \Cat_\infty$
that forgets the exact structure.
\end{definition}

\begin{definition}Let $(\caD, \psi) $ be an $\infty$-category with genuine pro-duality and $\caC, \caF \subset \caD$ wide subcategories.
The hermitian $\caQ$-construction is the full subfunctor $\caQ(\caD,\caC,\caF)$
of the functor $$\Delta^\op \to \Cat^{\mathrm{gp}}_{\infty}, [n] \mapsto \caD^{\Tw([n])}$$
such that for every $n \geq0$ the full subcategory with genuine pro-duality
$$\caQ(\caD,\caC,\caF)_{[n]} \subset \caD^{\Tw([n])}$$
is spanned by $\caQ(\caD,\caC,\caF)_{[n]}$,
where $\caD^{\Tw([n])}$ is the cotensor of \cref{notation_coten}. 
\end{definition}

\begin{remark}If $\caD$ is a Waldhausen $\infty$-category with genuine pro-duality, 
\cref{rmk:Q_exact} implies that the functor $$\caQ(\caD)\colon\Delta^\op \to \Cat^{\mathrm{gp}}_{\infty}$$
lifts to a functor $\Delta^\op \to \Wald^{\mathrm{gp}}_{\infty}.$
\end{remark}

\begin{remark}\label{constr:lambda}
There is a canonical equivalence 
$$\caH(\widetilde{[1] \times [1]}) \simeq \Lambda_0^2 \simeq \Tw([1])$$
of $\infty$-categories.
Let $\caD $ be a Waldhausen $\infty$-category with genuine pro-duality.
By \cref{maaap} there is a map $\caD^{\widetilde{[1] \times [1]}} \to \caD^{\Tw([1])}$
of Waldhausen $\infty$-categories with genuine pro-duality
whose underlying functor is restriction along the canonical embedding $ \Lambda_0^2 \subset [1] \times[1].$ This map restricts to a map
$$\lambda\colon \mathrm{Amb}(\caD) \to \caQ(\caD)_{[1]}$$ of Waldhausen $\infty$-categories with genuine pro-duality.
\end{remark}

\begin{proposition}\label{proqa}
Let $\caD $ be a Waldhausen $\infty$-category with genuine pro-duality.
The map of Waldhausen $\infty$-categories with genuine pro-duality $$\lambda\colon \mathrm{Amb}(\caD) \to \caQ(\caD)_{[1]}$$ 
of \cref{constr:lambda} is an equivalence.
\end{proposition}

\begin{proof}
We will verify that $\lambda$ induces an equivalence on underlying $\infty$-categories, on pro-dualities and on the genuine refinements.

That the functor $\lambda$ is an equivalence of $\infty$-categories is easy to observe. Indeed, recall that $\caQ(\caD)_{[1]} \subset \Fun(\Tw([1]),\caD) = \Fun({\Lambda^2_0},\caD)$ is the full subcategory
of spans $ X \to Y,X \to Z$, where $X\to Y$ is a fibration and $X \to Z$
is a cofibration.
The canonical functor $$\Fun([1]\times[1],\caD) \to \Fun(\Lambda^2_0,\caD)$$ restricts to an equivalence
$$\mathrm{Amb}(\caD) \to \caQ(\caD)_{[1]}$$ on underlying exact $\infty$-categories as noted in \cref{Q1_Amb_underlying_equiv}.

We now move to proving the equivalecne at the level of pro-dualities. Let $\sigma, \sigma'$ as below be ambigressive squares in $\caD$.
\begin{equation*}
\xymatrix{A \ar[r] \ar[d] & C \ar[d] \\
B \ar[r] & D,}\xymatrix{A' \ar[r] \ar[d] & C' \ar[d] \\
B' \ar[r] & D'}
\end{equation*}

We want to show that that the canonical map 

$$ \mathrm{Amb}(\caD)(\sigma, \sigma'^\dual) \longrightarrow \alpha_{\caQ(\caD)_{[1]}}(\lambda(\sigma),\lambda(\sigma'))$$
is an equivalence, where $\alpha_{\caQ(\caD)_{[1]}}$
is the pro-duality of $\caQ(\caD)_{[1]}$.
This map identifies with the canonical equivalence
\begin{align*}
\mathrm{Amb}(\caD)(\sigma, \sigma'^\dual)    &\simeq \caD(B,B'^\dual)\times_{\caD(A,B'^\dual) } \caD(A,D'^\dual) \times_{\caD(A,C'^\dual) } \caD(C,C'^\dual) \\
        &\simeq\caD(B,B'^\dual) \times_{\caD(A,B'^\dual) } (\caD(A,B'^\dual) \times_{\caD(A,A'^\dual)}\caD(A,C'^\dual) ) \times_{\caD(A,C'^\dual) } \caD(C,C'^\dual)\\
        &\simeq\caD(B,B'^\dual) \times_{\caD(A,A'^\dual) } \caD(C,C'^\dual)\\
        &\simeq\alpha_{\caQ(\caD)_{[1]}}(\lambda(\sigma),\lambda(\sigma')).
\end{align*}
For the last equivalence we use that by definition the pro-duality of $\caQ(\caD)_{[1]}$ is the pro-duality of the cotensor
$\caD^{\Lambda_0^2},$ which applies object-wise the pro-duality of $\caD$ underlying the duality.

It remains to see that $\lambda$ induces an equivalence on genuine refinements. For this we consider two genuine refinements  $H' \to \caH(\mathrm{Amb}(\caD))$ and $H'' \to \caH(\caQ(\caD)_{[1]})$,  and  $A \in \caH(\mathrm{Amb}(\caD))$. Let us consider $B\in \caH(\caQ(\caD)_{[1]})$ the image of $A$ under $\caH(\lambda)$.

Now, we can see $B$ as a functor $B\colon \Tw([1])\to \caH(\caD)$ via the inclusions
$$\caH(\caQ(\caD)_{[1]})\subset \caH(\caD^{\Tw([1])}) \simeq \Fun(\Tw([1]),\caH(\caD)).$$

There are canonical equivalences
$$ H'_A \simeq \Fun_{\caH(\caD)}(\caH(\widetilde{[1] \times [1]}) ,H), \hspace{1em}\text{ and }\hspace{1em}
H''_{B} \simeq  \Fun_{\caH(\caD)}(\Tw([1]),H)$$
and the induced map $H'_A \to H''_{B}$ identifies with the map
$$ \Fun_{\caH(\caD)}(\caH(\widetilde{[1] \times [1]}) ,H) \to \Fun_{\caH(\caD)}(\Tw([1]),H)$$ induced by the canonical
equivalence $\Tw([1])\simeq \Lambda^2_0 \simeq \caH(\widetilde{[1] \times [1]}),$ and therefore it is an equivalence itself.	
\end{proof}

\begin{proposition}\label{proposition_Segal}
Let $\caD $ be a Waldhausen $\infty$-category with genuine pro-duality. The simplicial object $\caQ(\caD)\colon \Delta^\op \to \Wald_\infty^\mathrm{gp}$ is Segal.
\end{proposition}

\begin{proof}
	
Let $n \geq 2$. We want to show that the Segal map
$$\theta\colon \caQ(\caD)_{[n]} \to \caQ(\caD)_{[1]} \times_{\caD} \dots \times_{\caD} \caQ(\caD)_{[1]}$$ is an equivalence of $\infty$-categories with genuine pro-duality.

By \cref{lemas} we already know that the map $\theta$ induces on underlying $\infty$-categories an equivalence. Moreover, by \cref{definition_exact} we know that this is an equivalence between exact $\infty$-categories.

Let $\caW_\caD \to \caD \times \caD$ the pro-duality of $\caD$,
and similar for $\caQ(\caD)_{[n]}.$
To prove this equivalence is an equivalence between pro-dualities, we want to see that the induced functor
$$ \caW_{\caQ(\caD)_{[n]} } \to \caW_{\caQ(\caD)_{[1]} } \times_{\caW_\caD} \dots
\times_{\caW_\caD} \caW_{\caQ(\caD)_{[1]} }$$ is an equivalence.
Now, by definition of the cotensor of $\infty$-categories with pro-dualities, this functor identifies with the canonical functor
$$ \caQ(\caW_{\caD},\caC,\caF)_{[n]} \to  \caQ(\caW_{\caD},\caC,\caF)_{[1]}  \times_{\caW_\caD} \dots 
\times_{\caW_\caD}  \caQ(\caW_{\caD},\caC,\caF)_{[1]}  $$
which is an equivalence by \cref{lemas}.

Similarly, at the level of \emph{genuine} pro-duality, we note that the map $\theta$ induces on genuine refinements the canonical functor
$$ \caQ(H,\caC,\caF)_{[n]} \to \caQ(H,\caC,\caF)_{[1]} \times_{H} \dots \times_{H} \caQ(H,\caC,\caF)_{[1]},$$ which is an equivalence by \cref{lemas}.
\end{proof}

\begin{corollary}
Let $\caD $ be a Waldhausen $\infty$-category with genuine duality. The Segal object $$\caQ(\caD)\colon \Delta^\op \to \Wald_\infty^\mathrm{gp}$$ induces a Segal object $$\caQ(\caD)\colon \Delta^\op \to \Wald_\infty^\mathrm{gd}. $$
\end{corollary}

\begin{proof}
In view of the Segal condition, \cref{proposition_Segal},
it is enough to see that the Waldhausen $\infty$-category with genuine pro-duality $\caQ(\caD)_{[1]}$ is a Waldhausen $\infty$-category with genuine duality.
By \cref{proqa} there is an equivalence $\caQ(\caD)_{[1]} \simeq \mathrm{Amb}(\caD)$ of Waldhausen $\infty$-categories with genuine pro-duality.
The Waldhausen $\infty$-category with genuine pro-duality $\mathrm{Amb}(\caD)$ is in fact a Waldhausen $\infty$-category with genuine duality by \cref{notation_ambwald}.
\end{proof}

\section{Comparison between real \texorpdfstring{$S$}\ - and hermitian \texorpdfstring{$Q$}\ -constructions}\label{sec:comparison}

The content of this section is a comparison between the $K$-theory spaces defined in \cite{Calmes_etal2} and \cite{realKthHSV}. We first compare the $S_\bullet$-construction of \cite{realKthHSV} and the hermitian $Q_\bullet$-construction of \cref{hermq}; see \cref{comp}. This, together with the coincidence between $Q$-constructions explained in \cref{Qconstructions_coincide} allow us to conclude the sought comparison in \cref{cor:Ktheory_coincide}.

\begin{notation}\label{rho_Qcosntruction}
For every $n \geq 0$ we consider the canonical embeddings $$ \alpha\colon [n] \hookrightarrow [n]*[n]^\op, \ \beta\colon[n]^\op \hookrightarrow [n]*[n]^\op.$$ Let $\rho_n$ be the functor
$$\Tw([n]) \to \caH(\Ar([n]*[n]^\op))$$
between posets that sends a morphism $i \leq j$ of $[n]$ to
the morphism $\alpha(i) \leq \beta(j) $ of $[n]*[n]^\op$
equipped with the unique morphism $$(\alpha(i) \leq \beta(j)) \to (\alpha(i) \leq \beta(j))^\dual = (\beta(j)^\dual \leq \alpha(i)^\dual)= (\alpha(j) \leq \beta(i))$$ in $\Ar([n]*[n]^\op)$ that exists because $\alpha(i) \leq \alpha(j) $ and $\beta(j) \leq \beta(i).$
\end{notation}

\begin{construction}\label{constr_map}
For every Waldhausen $\infty$-category with genuine duality $\caC$
let $\gamma_\caC^n$ be
the composition
$$ \caC^{\Ar([n]*[n]^\op)} \longrightarrow \caC^{\caH(\Ar([n]*[n]^\op))} \longrightarrow \caC^{\Tw([n])} $$
in $\Wald_\infty^\gd.$ The map $\gamma_\caC^n$ restricts to a map
$$\theta^n_\caC: S(\caC)_{[n]*[n]^\op} \to \caQ(\caC)_{[n]}$$
in $\Wald_\infty^\gd$. Since $\rho_n$ in \cref{rho_Qcosntruction} is natural in $[n] \in \Delta$, the map $\gamma_\caC^n $ 
and so also $\theta^n_\caC$ are natural in $[n] \in \Delta$. Hence the latter organize to a map $$\theta_\caC\colon S(\caC)\circ e \to \caQ(\caC)$$ of simplicial Waldhausen $\infty$-categories with genuine duality, where $e$ is the edgewise subdivision.
\end{construction}

\begin{theorem}\label{comp}
Let $\caC$ be a Waldhausen  $\infty$-category with genuine duality. The map $$\theta_\caC\colon S(\caC)\circ e \to \caQ(\caC)$$ of simplicial Waldhausen $\infty$-categories with genuine duality is an equivalence.	
\end{theorem}

\begin{proof}
We know that both source and target of $\theta_\caC$ are Segal objects. Thus, it is enough to prove that for $n=0,1$ the map $$\theta_{[n]}\colon S(\caC)_{[n]*[n]^\op} \to \caQ(\caC)_{[n]}$$ of exact $\infty$-categories with genuine duality is an equivalence.

For $n=0$ this is clear, as the map $\theta_{[0]}$ identifies with the identity of $\caC.$ For $n=1$, we recall that by \cite[Proposition 8.13]{realKthHSV}, the embedding  $\widetilde{[1] \times [1]} \subset \Ar([3]) $ of $\infty$-categories with duality sending 
$$(0,0)\mapsto (0,2), \hspace{0.8em} (0,1)\mapsto (0,3),\hspace{0.8em} (1,0)\mapsto (1,2),\hspace{0.5em}\text{and}\hspace{0.5em}(1,1)\mapsto(1,3) $$
induces an equivalence $S(\caC)_{[3]} \simeq \mathrm{Amb}(\caC)$ of exact $\infty$-categories with genuine duality. Moreover, by \cref{proqa} there is an equivalence of exact $\infty$-categories with genuine duality $\mathrm{Amb}(\caC) \to \caQ(\caC)_{[1]}$, which is the restriction of the canonical map
$$ \caC^{\widetilde{[1] \times [1]}} \to \caC^{\caH(\widetilde{[1] \times [1]})}
\simeq \caC^{\Tw([1])}$$
induced by the equivalence $ \caH(\widetilde{[1] \times [1]}) \simeq \Lambda_0^2 \simeq \Tw([1])$ of \cref{constr:lambda}. Finally, the map $\theta_{[1]}$ factors in $\Exact_\infty^\gd$ as below. 
\begin{tz}
\node[](1) {$S(\caC)_{[3]}$}; 
\node[below of=1,xshift=2cm,yshift=0cm](3) {$\mathrm{Amb}(\caC)$};
\node[above of=3,xshift=2cm,yshift=-0cm](2) {$\caQ(\caC)_{[1]}$}; 

\draw[->] (1) to node[above,la]{$\theta_{[1]}$} (2);
\draw[->] (3) to node[right,la,xshift=-2pt,yshift=-6pt]{$\simeq$} (2);
\draw[->] (1) to node[left,la,yshift=-4pt]{$\simeq$} (3);
\end{tz}
\end{proof}

\section{Comparison of hermitian \texorpdfstring{$K$}\ -theory spaces}\label{sec:comparison_Kth}

We use the tools constructed before to prove an equivalence between the real $K$-theory spaces of \cite{Calmes_etal1} and \cite{realKthHSV} when both frameworks make sense.

\begin{definition}
Let $\caC$ be a stable $\infty$-category. A quadratic functor on $\caC$ is a 2-excisive functor $\caC^\op \to \Sp$. 
\end{definition}

\begin{remark}

By definition a quadratic right fibration classifies a quadratic functor $\caC^\op \to \Sp_{\geq0}$ that by \cite[6.14.(2)]{realKthHSV} is the connective cover of a unique quadratic functor $\caC^\op \to \Sp$.

\end{remark}

\begin{definition}

A quadratic functor is non-degenerate if its polarization is non-degenerate.

\end{definition}

\begin{definition}
\begin{enumerate}
\item A Poincar\'e $\infty$-category is a pair $(\caC, \phi),$
where $\caC$ is a stable $\infty$-category and
$\phi$ is a non-degenerate quadratic functor.
\item A map of Poincar\'e $\infty$-categories $(\caC, \phi) \to (\caD, \psi)$ is a pair $(F, \alpha),$ where $F: \caC \to \caD$ is an exact functor and $\alpha$ is a natural transformation $ \phi \to \psi \circ F^\op$ 
that induces on polarizations a duality preserving functor.
\end{enumerate}
 We denote by $\Cat_\infty^p$ be the $\infty$-category of small Poincar\'e $\infty$-categories.
\end{definition}

\begin{theorem}\label{equiv_qu_gd}\label{eqiv_Poincare_stablegd}
There is an equivalence of $\infty$-categories
$$\Cat_\infty^p \simeq \St^{\mathrm{gd}}$$
left tensored over $\Spc^{C_2}$.
\end{theorem}

\begin{proof}
By \cite[Theorem 6.27]{realKthHSV}, there is a canonical equivalence when know that there is a canonical equivalence between the $\infty$-categories of stable $\infty$-categories with genuine duality and that of quadratic right fibrations
$$\St^{\mathrm{gd}} \simeq \mathrm{QuR}.$$
By definition (non-degenerate) quadratic right fibrations classify (non-degenerate) quadratic functors to $\Sp_{\geq0}$,
which by \cite[Proposition 6.14.(2)]{realKthHSV}
uniquely extend to (non-degenerate) quadratic functors to $\Sp.$
\end{proof}

\begin{remark}\label{Qconstructions_coincide}
There is a commutative square
$$\begin{xy}
\xymatrix{\mathrm{St}^\mathrm{gd}  \ar[d]^{\caQ} \ar[rr]^{\xi}
&& \Cat_\infty^p\ar[d]^{\caQ'} 
\\
\Fun(\Delta^\op,\mathrm{St}^\mathrm{gd}) \ar[rr]^{\Fun(\Delta^\op,\xi)} &&\Fun(\Delta^\op,\Cat_\infty^p),}
\end{xy}$$
where $\caQ'$ is the hermitian $Q$-construction of \cite[Definition 2.2.1.]{Calmes_etal2} and $\caQ$ is the hermitian $Q$-construction of \cref{hermq}. To see this, we note that the equivalence $\xi\colon \mathrm{St}^\mathrm{gd} \simeq \Cat_\infty^p$ of \cref{equiv_qu_gd} preserves cotensors with small $\infty$-categories. Since
 both $\caQ(\caC)_n$ and $\caQ'(\caD)_n$ are defined respectively as the restrictions of the cotensors
$\caC^{\Tw([n])}$ in $\mathrm{St}^\mathrm{gd}$ and $ \caD^{\Tw([n])}$ in $\Cat_\infty^p$ (see  \cite[Definition 2.2.1.]{Calmes_etal2} and \cite[Definition 6.3.4.]{Calmes_etal1}).
\end{remark}

\begin{notation}
Let $(\caC,\phi)$ be a Poincar\'e $\infty$-category.
Let $(\caC,\phi)[1]\coloneqq (\caC,(-)[1] \circ \phi) $, which is again a Poincar\'e $\infty$-category,
where $(-)[1]\colon \Sp \to \Sp$ is the suspension functor.
\end{notation}

The following definition, adapted to our notation for non-degenerate hermitian objects, is \cite[Corollary 4.1.5.]{Calmes_etal2}.

\begin{definition}\label{def:GW} Let $(\caC,\phi)$ be a Poincar\'e $\infty$-category. The Grothendieck-Witt space of  $(\caC,\phi)$  is
$$\GW(\caC)\coloneqq \Omega |(-)^{C_2} \circ (-)^\simeq \circ \caQ((\caC,\phi)[1]) |.$$
\end{definition}

Let $(-)_+\colon \Spc \to \Spc_*$ be the left adjoint of the forgetful functor that adds a disjoint base point.

\begin{definition}
Let $S^{1,1} $ be the cofiber of the map of pointed genuine $C_2$-spaces
$$ (C_2)_+ \to *_+ \to S^{1,1}. $$ 
\end{definition}

\begin{remark}
Since the forgetul functor $\Spc^{C_2} \to \Spc$ preserves colimits,
the underlying space of $S^{1,1}$ is the 1-sphere.
The $C_2$-action on $S^{1,1}$ is the sign representation.
    
\end{remark}

\begin{notation}
Let $\Omega^{1,1}\colon \Spc^{C_2} \to \Spc^{C_2} $
the functor cotensoring with $S^{1,1}.$
\end{notation}

The following definition is \cite[Definition 9.21.]{realKthHSV}:

\begin{definition} Let $(\caC,\phi)$ be a Waldhausen $\infty$-category with genuine duality. The real $K$-theory space of  $(\caC,\phi)$  is
$$\KR(\caC)\coloneqq \Omega^{1,1} | (-)^\simeq \circ S(\caC) |.$$
\end{definition}

To establish the first part of the comparison, we need the following result that relates, for any genuine $C_2$-space, the fixed points of $\Omega^{1,1}(X)$ with those of $X$.

\begin{lemma}[{\cite[Corollary 9.5]{realKthHSV}}] For every genuine $C_2$-space $X$, there is a fiber sequence of spaces
$$\Omega^{1,1}(X)^{C_2}\to X^{C_2} \to X.$$
\end{lemma}

\begin{theorem}\label{cor:Ktheory_coincide}
Let $\caC$ be a stable $\infty$-category with genuine duality and $\caC'$ the corresponding Poincar\'e $\infty$-category via the equivalence of \cref{eqiv_Poincare_stablegd}. There is an equivalence $$\KR(\caC)^{C_2} \to \GW(\caC')$$ of grouplike $\bE_\infty$-spaces.
\end{theorem}

\begin{proof}
The map  $$ S(\caC) \circ e \longrightarrow \caQ(\caC) $$ of simplicial $\infty$-categories with genuine duality of \cref{constr_map}, which is an equivalence by \cref{comp}, together with \cref{Qconstructions_coincide}
induce an equivalence $$ (-)^\simeq \circ S(\caC) \circ e \longrightarrow (-)^\simeq \circ\caQ(\caC') $$ of simplicial spaces with genuine $C_2$-action. The latter map induces on $C_2$-fixed points of  geometric realizations the following map of grouplike $\bE_\infty$-spaces 
$$
\KR(\caC)^{C_2}=(\Omega^{1,1}|(-)^\simeq \circ S(\caC) \circ e|)^{C_2} \longrightarrow (\Omega^{1,1}|(-)^\simeq \circ\caQ(\caC')|)^{C_2}  \simeq $$
$$ \fib(|(-)^\simeq \circ\caQ(\caC')|^{C_2} \to |(-)^\simeq \circ\caQ(\caC')|) \simeq $$$$ \fib(|(-)^{C_2} \circ (-)^\simeq \circ\caQ(\caC')| \to |(-)^\simeq \circ\caQ(\caC')|) \simeq $$
$$ \Omega |(-)^{C_2} \circ (-)^\simeq \circ \caQ((\caC,\phi)[1]) |= \GW(\caC').$$ 
\end{proof}

\section{Comparison of real \texorpdfstring{$K$}-theory spectra.}
The aim of this section is to upgrade the comparison between the two approaches by proving the following result. 

\begin{theorem}\label{theorem_compsp}
Let $\caC$ be a stable $\infty$-category with genuine duality and $\caC'$ the corresponding Poincar\'e $\infty$-category via the equivalence in \cref{eqiv_Poincare_stablegd}. There is a canonical equivalence $$\KR(\caC) \to \KR'(\caC')$$ of genuine $C_2$-spectra.
\end{theorem}

Before diving into the proof, we recall both constructions of the real $K$-theory spectra.

\subsection{Construction of Calm\`es et al.}

\begin{notation}

Let $\caO_{C_2}$ be the orbit category of $C_2.$
\end{notation}

\begin{definition}
A genuine $C_2$-category is a functor $\caO_{C_2}^\op \to \Cat_\infty.$ We write $$ \Cat_\infty^{C_2}\coloneqq \Fun(\caO^\op_{C_2},\Cat_\infty)$$ for
the $\infty$-category of genuine $C_2$-categories.
\end{definition}

\begin{example}\label{example_span}
By \cite[Definition 4.10.]{nardin2016parametrized} there is a genuine $C_2$-category
$ \overline{\Span(\Fin[C_2])}$
such that $$ \overline{\Span(\Fin[C_2])}(*) \simeq \Span(\Fin[C_2]) $$
and $$ \overline{\Span(\Fin[C_2])}(C_2) \simeq \Span(\Fin[C_2]/C_2) $$
and the canonical functor
$$ \overline{\Span(\Fin[C_2])}(*) \simeq \Span(\Fin[C_2]) \to  \overline{\Span(\Fin[C_2])}(C_2) \simeq \Span(\Fin[C_2]/C_2) $$
is induced by the functor
$$ (-) \times C_2\colon\Fin[C_2] \to \Fin[C_2]/C_2. $$ 
\end{example}

\begin{remark}
By \cite[Lemma A.59.]{realKthHSV} there is a canonical forgetful functor
$$ \nu\colon\Cat_\infty^{\Spc^{C_2}} \to \Cat_\infty^{C_2}.$$

For every real $\infty$-category $\caC$ the $\infty$-category
$\nu(\caC)(*)$ is the $\infty$-category that carries the enrichment
and $\nu(\caC)(C_2) $ is the underlying $\infty$-category.
\end{remark}

In what follows we recall notions of preadditivity introduced in \cite{Calmes_etal1} and \cite{realKthHSV} before comparing them in \cref{rmk:compare_preadditivity}. The next definition is the content of Definition 7.4.9 and Lemma 7.4.4 in \cite{Calmes_etal1} combined\textemdash note that we call here $C_2$-preadditive what is there called $C_2$-semiadditive.

\begin{definition}\label{def:C2_preadditive}
A genuine $C_2$-category $\caC$ is $C_2$-preadditive if the $\infty$-categories
$\caC(*), \caC(C_2)$ are preadditive and the canonical functor
$g\colon \caC(*) \to \caC(C_2)$ admits a left adjoint $f$ and a right adjoint $h$ such that the following holds:

\begin{enumerate}
\item For every $X \in \caC(C_2)$ the morphism
$$ X \coprod \tau(X) \to g(f(X))$$ corresponding to the codiagonal morphism
$$ f(X \coprod \tau(X)) \simeq f(X) \coprod f(\tau(X)) \simeq f(X) \coprod f(X) \to f(X)$$ is an equivalence.

\item For every $X \in \caC(C_2)$ the morphism
$$ g(h(X)) \to X \prod \tau(X) $$ corresponding to the diagonal morphism
$$ h(X) \to h(X \prod \tau(X))\simeq h(X) \prod h(\tau(X))\simeq h(X) \prod h(X) $$ is an equivalence.

\item For every $X \in \caC(C_2)$ the morphism
$$f(X) \to h(X) $$ corresponding to the summand inclusion
$$ X \to g(h(X)) \simeq X \oplus \tau(X)$$ is an equivalence.

\end{enumerate}
\end{definition}

\begin{definition}
Let $\caC, \caD$ be genuine $C_2$-categories such that
$\caC(*), \caC(C_2)$ admit finite coproducts and the functor
$\caC(*) \to \caC(C_2)$ preserves finite coproducts and admits a left adjoint, and the same for $\caD.$ A genuine $C_2$-functor $\caC \to \caD$ is $C_2$-preadditive if
the induced functors
$$\caC(*) \to \caD(*), \ \caC(C_2) \to \caD(C_2)$$ preserve finite coproducts and commute with the left adjoints of the functors
$$ \caC(*) \to \caC(C_2), \ \caD(*) \to \caD(C_2).$$
    
\end{definition}

\begin{remark}

Let $\caC, \caD$ be $C_2$-preadditive genuine $C_2$-categories.
For a $C_2$-preadditive genuine $C_2$-functor $\caC \to \caD$ 
the induced functors
$$\caC(*) \to \caD(*), \ \caC(C_2) \to \caD(C_2)$$ also preserve finite products and commute with the right adjoints of the functors
$$ \caC(*) \to \caC(C_2), \ \caD(*) \to \caD(C_2).$$
    
\end{remark}

Now we recall the notion of genuine preadditivity of \cite[Definition 5.13]{realKthHSV}.

\begin{definition}\label{def:genuine_preadd}
We call a real $\infty$-category $\caC$ genuine preadditive if 
\begin{enumerate}
\item the underlying (that is, forgetting the enrichment) $\infty$-category of $\caC$ is preadditive,
\item $\caC$ admits a zero object, finite products and finite coproducts (in the $\Spc^{C_2}$-enriched sense) and tensors and cotensors with $C_2$, and
\item for every $X \in \caC$, the natural map $C_2 \otimes X \to X^{C_2}$ 
corresponding to the $C_2$-equivariant map
$$C_2 \times C_2 \longrightarrow * \xrightarrow{\id_\X} \Mor_{\caC}(X, X) $$ is an equivalence.
\end{enumerate}
\end{definition}

\begin{remark}\label{rmk:compare_preadditivity} In the following items we describe why \cref{def:C2_preadditive} and \cref{def:genuine_preadd} coincide.
\begin{enumerate}
\item A real $\infty$-category $\caC$ admits finite conical (co)products 
if and only if $\nu(\caC)(*), \nu(\caC)(C_2)$ admit finite (co)products and
the functor $\nu(\caC)(*) \to \nu(\caC)(C_2)$ preserves finite (co)products.
 
\item A real $\infty$-category $\caC$ admits (co)tensors with $C_2$
if and only if the functor 
$$\nu(\caC)(*) \to \nu(\caC)(C_2)$$ admits a left (right) adjoint. 
If such exist, the canonical functor
$$\nu(\caC)(*) \to \nu(\caC)(C_2)$$ followed by the left (right) adjoint assigns the (co)tensor with $C_2$.

\item Hence a real $\infty$-category $\caC$ is genuine preadditive if and only if the underling genuine $C_2$-category $\nu(\caC)$ is $C_2$-preadditive.

\item Hence a real functor between real $\infty$-categories having finite conical coproducts and tensors with $C_2$ preserves finite conical coproducts and tensors with $C_2$ if and only if the underling genuine $C_2$-functor is $C_2$-preadditive.
\end{enumerate}
\end{remark}

In the following we recall the generalization of the idea that every genuine $C_2$-spectrum has an underlying spectral Mackey functor, which is a finite products preserving functor $\Span(\Fin[C_2]) \to \Sp$ (see \cite{barwick2017spectral,guillou2024models}) \cref{rho} is \cite[Proposition 7.4.16.]{Calmes_etal1}, we include a proof as it aids to prove \cref{rho2} which will be crucial for the comparison we seek.

\begin{proposition}\label{prop_Mackey} Let $\caC$ be a $C_2$-preadditive
genuine $C_2$-category.

\begin{enumerate}
\item\label{rho} There is a canonical functor
$$ \rho_\caC: \caC(*) \to \Fun^{\prod}(\Span(\Fin[C_2]),\caC(*))$$
such that the composition with the functor 
$$\Fun^{\prod}(\Span(\Fin[C_2]),\caC(*)) \to \caC(*) $$ evaluating
at $*$ factors as the functor $\caC(*) \to \caC(C_2)$ followed by its  left adjoint.

\vspace{1mm}
\item\label{rho2} Let $\caD$ be a $C_2$-preadditive genuine $C_2$-category and $\caC \to \caD$ a $C_2$-preadditive genuine $C_2$-functor.
There is a canonical commutative square:

$$\begin{xy}
\xymatrix{\caC(*) \ar[d]^{\rho_\caC} \ar[rrr]^{\phi(*)}
&&& \caD(*) \ar[d]^{\rho_\caD}
\\
\Fun^{\prod}(\Span(\Set[C_2]),\caC(*)) \ar[rrr]^{\phi(*)_*} &&& \Fun^{\prod}(\Span(\Set[C_2]),\caD(*)).}
\end{xy}$$
\end{enumerate}
    
\end{proposition}

\begin{proof}

(1): For every genuine $C_2$-categories $\caC, \caD$ let $\Fun^{C_2}(\caC,\caD)$
be the morphism $\infty$-category in $\Cat_\infty^{C_2}.$
The functor $$(-)^{C_2}: \Cat_\infty^{C_2} \to \Cat_\infty$$ taking $C_2$-fixed points gives rise to a functor
\begin{equation}\label{mapos}
\Fun^{C_2}(\caC,\caD) \to \Fun(\caC(*),\caD(*)).\end{equation}

We use the genuine $C_2$-category $ \overline{\Span(\Fin[C_2])}$ of \cref{example_span}.

Let $$ \Fun^{C_2, \prod}(\overline{\Span(\Fin[C_2])},\caC) \subset \Fun^{C_2}(\overline{\Span(\Fin[C_2])},\caC) $$
be the full subcategory spanned by the $C_2$-preadditive genuine $C_2$-functors.

The functor (\ref{mapos}) restricts to a functor $$\xi: \Fun^{C_2, \prod}(\overline{\Span(\Fin[C_2])},\caC) \to \Fun^{\prod}(\Span(\Fin[C_2]),\caC(*)).$$

Evaluation at the final $C_2$-set gives a functor $$\ev_*: \Fun^{\prod}(\Span(\Fin[C_2]),\caC(*)) \to \caC(*).$$

By \cite[Proposition 5.11., Theorem 6.5.]{nardin2016parametrized} the composition 
$$\Fun^{C_2, \prod}(\overline{\Span(\Fin[C_2])},\caC) \xrightarrow{\xi} \Fun^{\prod}(\Span(\Fin[C_2]),\caC(*)) \to \caC(*) $$
is an equivalence.
Let $\rho_\caC$ be the composition 
$$\caC(*) \simeq \Fun^{C_2, \prod}(\overline{\Span(\Fin[C_2])},\caC) \xrightarrow{\xi} \Fun^{\prod}(\Span(\Fin[C_2]),\caC(*)).$$

(2): Since $\phi$ preserves $C_2$-coproducts, there are canonical 
commutative squares
$$\begin{xy}
\xymatrix{\Fun^{C_2, \prod}(\overline{\Span(\Fin[C_2])},\caC) \ar[d]^{\xi} \ar[rrr]^{\phi_*}
&&& \Fun^{C_2, \prod}(\overline{\Span(\Fin[C_2])},\caD) \ar[d]^{\xi}
\\
\Fun^{\prod}(\Span(\Set[C_2]),\caC(*)) \ar[rrr]^{\phi(*)_*}
&&& \Fun^{\prod}(\Span(\Set[C_2]),\caD(*)) }
\end{xy}$$
and 
$$\begin{xy}
\xymatrix{\Fun^{\prod}(\Span(\Set[C_2]),\caC(*)) \ar[d]^{\ev_*} \ar[rrr]^{\phi(*)_*}
&&& \Fun^{\prod}(\Span(\Set[C_2]),\caD(*)) \ar[d]^{\ev_*}
\\
\caC(*) \ar[rrr]^{\phi(*)} &&&  \caD(*)}
\end{xy}$$

This proves (2).
\end{proof}

\begin{remark}\label{rmk:models_gensp} We recall the equivalence between the standard model and the Mackey functor model of genuine $C_2$-spectra.
The $\infty$-category $\Sp^{C_2}$ is genuine preadditive by \cite[Lemma 10.19]{realKthHSV}. Hence there is a functor
$$\rho_{\nu(\Sp^{C_2})}\colon \Sp^{C_2} \to \Fun^{\prod}(\Span(\Set[C_2]),\Sp^{C_2}).$$

Postcomposition with the functor $(-)^{C_2}\colon \Sp^{C_2} \to \Sp$ induces a functor
$$ \Fun^{\prod}(\Span(\Set[C_2]),\Sp^{C_2}) \to \Fun^{\prod}(\Span(\Set[C_2]),\Sp).$$

Now, taking the composition we get a functor $$\Sp^{C_2} \xrightarrow{\rho_{\nu(\Sp^{C_2})}} \Fun^{\prod}(\Span(\Set[C_2]),\Sp^{C_2})\to \Fun^{\prod}(\Span(\Set[C_2]),\Sp),$$
which is an equivalence, see \cite{barwick2017spectral}, \cite{guillou2024models}.
\end{remark}

\begin{corollary}\label{cor_form}

Let $\caC$ be a $C_2$-preadditive genuine $C_2$-category and $\phi: \caC \to \nu(\Sp^{C_2})$ a $C_2$-preadditive genuine $C_2$-functor.
There is a canonical commutative square as below
$$\begin{xy}
\xymatrix{\caC(*) \ar[d]^{\rho_\caC} \ar[rrr]^{\phi(*)}
&&& \Sp^{C_2} \ar[d]^{\simeq}
\\
\Fun^{\prod}(\Span(\Set[C_2]),\caC(*)) \ar[rrr]^{((-)^{C_2} \circ \phi(*))_*} &&& \Fun^{\prod}(\Span(\Set[C_2]),\Sp).}
\end{xy}$$

\end{corollary}

\begin{proof}
By \cref{prop_Mackey} there is a canonical commutative square:
$$\begin{xy}
\xymatrix{\caC(*) \ar[d]^{\rho_\caC} \ar[rrr]^{\phi}
&&& \Sp^{C_2} \ar[d]^{\rho_{\nu(\Sp^{C_2})}}
\\
\Fun^{\prod}(\Span(\Set[C_2]),\caC(*)) \ar[rrr]^{\phi_*} &&& \Fun^{\prod}(\Span(\Set[C_2]),\Sp^{C_2}).}
\end{xy}$$
We prolong this square with the functor
$$\Fun^{\prod}(\Span(\Set[C_2]),\Sp^{C_2}) \to \Fun^{\prod}(\Span(\Set[C_2]),\Sp)$$ induced by the functor $(-)^{C_2}\colon \Sp^{C_2} \to \Sp$ taking $C_2$-fixed points to obtain the result. Note that by \cref{rmk:models_gensp}, the vertical right map is an equivalence.
\end{proof}

In \cite{Calmes_etal2} the authors define
real $K$-theory the following way:

\begin{definition}\label{def_realK_9authors}
The real $K$-theory functor
$$\KR'\colon \Cat^p \to \Fun^{\prod}(\Span(\Set[C_2]),\Sp) \simeq \Sp^{C_2} $$
is the composition
$$ \Cat_\infty^p \xrightarrow{\rho_{\Cat_\infty^p}} \Fun^{\prod}(\Span(\Set[C_2]),\Cat_\infty^p) \xrightarrow{\GW_*} \Fun^{\prod}(\Span(\Set[C_2]),\Sp)$$
where $\GW$ is as in \cref{def:GW}.

\end{definition}

\subsection{Our construction} We briefly recall some results needed to understand our definitition of real $K$-theory from \cite{realKthHSV}, which we recall in \cref{def:KR_ours}

\begin{proposition}[{\cite[Theorem A.50]{realKthHSV}}]\label{prop_preaddlift}
Let $\caC$ be a preadditive real $\infty$-category and $\caD$ a real $\infty$-category that admits finite conical coproducts.
The induced functor
$$ \Fun^{\Spc^{C_2}, \prod}(\caC,\Mon_{\bE_\infty}(\caD)) \to \Fun^{\Spc^{C_2}, \prod}(\caC,\caD) $$
is an equivalence.

    
\end{proposition}

\begin{definition}

Let $\caC$ be a preadditive $\infty$-category and $X \in \caC$.
The shear map $$X \times X \to X \times X$$ on $X$
is the map whose projection to the first component is the fold map
$$X \times X \simeq X \coprod X \to X $$ and whose projection to the second component is the projection to the second component.
    
\end{definition}

\begin{corollary}\label{cor_liftgrp}

Let $\caC$ be a preadditive real $\infty$-category.
The induced functor
$$ \Fun^{\Spc^{C_2}, \prod}(\caC,\Fun(\caO^\op_{C_2}, \Sp_{\geq 0})) \to \Fun^{\Spc^{C_2}, \prod}(\caC,\Spc^{C_2}) $$
is fully faithful and the essential image precisely consists of the real functors $\phi: \caC \to \Spc^{C_2}$ preserving finite products 
such that for every $X \in \caC$ the shear map $$ \phi(X) \times \phi(X) \to \phi(X) \times \phi(X) $$ on $\phi(X)$  is an equivalence.
    
\end{corollary}

\begin{proof}
By \cref{prop_preaddlift} every real functor $\phi: \caC \to \Spc^{C_2}$ preserving finite products 
such that for every $X \in \caC$ the shear map $$ \phi(X) \times \phi(X) \to \phi(X) \times \phi(X) $$ on $\phi(X)$  is an equivalence, 
uniquely lifts to a real functor $$\caC \to \Grp_{\bE_\infty}(\Spc^{C_2}) \subset \Mon_{\bE_\infty}(\Spc^{C_2}) $$ preserving finite products. Moreover, there is a canonical equivalence $$\Fun(\caO^\op_{C_2}, \Sp_{\geq 0}) \simeq \Fun(\caO^\op_{C_2},  \Grp_{\bE_\infty}(\Spc)) \simeq \Grp_{\bE_\infty}(\Spc^{C_2}),$$
which concludes the proof.
\end{proof}

\begin{remark}
There is a canonical equivalence
$$\Fun(\caO^\op_{C_2}, \Sp) \simeq \Sp(\Fun(\caO^\op_{C_2}, \Spc) = \Sp(\Spc^{C_2}) = \Sp_{S^{1,0}}(\Spc^{C_2}).$$ In \cite{realKthHSV} we used the terminology on the right, but for the purposes of this note, we use the terminology on the left.
\end{remark}

\begin{definition}\label{def:KR_ours}

The real functor
$$\KR : \Wald_\infty^\gd \to \Fun(\caO^\op_{C_2}, \Sp_{\geq 0}) $$
is the unique real functor 
of \cref{cor_liftgrp} preserving finite products that lifts the real functor
$$\KR: \Wald_\infty^\gd \to \Spc^{C_2} $$
of \cite[Definition 9.21.]{realKthHSV}, which preserves finite products and sends any Waldhausen
$\infty$-category with genuine duality to a genuine $C_2$-space whose shear map is an equivalence \cite[Lemma 10.32.]{realKthHSV}.

\end{definition}

Next we recall how we lift the real $K$-theory real functor
$$\KR\colon \Wald_\infty^\gd \to \Fun(\caO_{C_2}^\op,\Sp) $$
to the real $\infty$-category $\Sp^{C_2}$ of genuine $C_2$-spectra.

\begin{definition}

A real $\infty$-category is reduced if it admits a zero object.
A real functor is reduced if it preserves the zero object.
    
\end{definition}

\begin{remark}
Every reduced real $\infty$-category uniquely refines to a $\Spc^{C_2}_*$-enriched $\infty$-category.
Every reduced real functor uniquely refines to a $\Spc^{C_2}_*$-enriched functor.
\end{remark}

The next definition is \cite[Definition 9.6.]{realKthHSV}:

\begin{definition}
Let $\caC$ be a reduced real $\infty$-category that admits tensors with
$S^{1,1} $ and $\caD$ a reduced real $\infty$-category that admits cotensors with $S^{1,1}.$

A real functor $\caC \to \caD$ is genuine excisive
if for every $X \in \caC$ the morphism
$$ F(X) \to F(S^{1,1} \otimes X)^{S^{1,1}} $$
corresponding to the morphism 
$$S^{1,1} \otimes F(X) \to F(S^{1,1} \otimes X) $$
is an equivalence, where tensors and cotensors are formed with respect to the $\Spc^{C_2}_*$-enrichment.
    
\end{definition}

The following is \cite[Proposition 10.20.]{realKthHSV}:

\begin{proposition}
Let $\caC$ be a reduced real $\infty$-category that admits tensors with
$S^{1,1} $.
\begin{enumerate}
\item Every genuine excisive real functor $$\phi: \caC \to \Fun(\caO_{C_2}^\op,\Sp) $$ uniquely lifts to a 
genuine excisive real functor $$\phi': \caC \to \Sp^{C_2}.$$

\item If $\phi$ preserves finite products, then also $\phi'$
preserves finite products.

\item If $\phi$ preserves cotensors with finite $C_2$-spaces, then also $\phi'$
preserves cotensors with finite $C_2$-spaces.
\end{enumerate}
\end{proposition}

The next definition is \cite[Definition 9.21.]{realKthHSV}:

\begin{definition}
We define the real functor
$$\KR\colon \Wald_\infty^\gd \to \Spc^{C_2} $$
as the unique genuine excisive real functor
$$ \Wald_\infty^\gd \to \Spc^{C_2} $$
that presererves finite products and cotensors with finite $C_2$-sets
lifting the genuine excisive real functor
$$\KR : \Wald_\infty^\gd \to \Fun(\caO_{C_2}^\op,\Sp) $$
that preserves finite products and cotensors with finite $C_2$-sets.
\end{definition}

\subsection{Comparison}

\begin{proposition}\label{cor_compa}
There is a canonical commutative square:
$$\begin{xy}
\xymatrix{\Wald_\infty^\gd \ar[d]^{\rho_{\Wald_\infty^\gd}} \ar[rrr]^{\KR}
&&& \Sp^{C_2} \ar[d]^{\simeq}
\\
\Fun^{\prod}(\Span(\Set[C_2]),\Wald_\infty^\gd) \ar[rrr]^{\KR^{C_2}_*} &&& \Fun^{\prod}(\Span(\Set[C_2]),\Sp).}
\end{xy}$$

that restricts to commutative square below when we consider stable $\infty$-categories.
$$\begin{xy}
\xymatrix{\St^\gd \ar[d]^{\rho_{\St^\gd}} \ar[rrr]^{\KR}
&&& \Sp^{C_2} \ar[d]^{\simeq}
\\
\Fun^{\prod}(\Span(\Set[C_2]),\St^\gd) \ar[rrr]^{\KR^{C_2}_*} &&& \Fun^{\prod}(\Span(\Set[C_2]),\Sp).}
\end{xy}$$
\end{proposition}

\begin{proof} 
By \cite[Proposition 5.20., Remark 7.10.]{realKthHSV} the real $\infty$-category $\Wald_\infty^\gd$ is genuine preadditive. Furthermore, since $\Sp^{C_2}$ is genuine preadditive, we find that 
$$\KR \colon \Wald_\infty^\gd \to \Sp^{C_2} $$ also preserves finite coproducts and tensors with finite $C_2$-sets. Thus we can apply \cref{cor_form} to obtain the desired result.
\end{proof}

\begin{proof}[Proof of \cref{theorem_compsp}]

By \cref{equiv_qu_gd} there is an equivalence
$$\St^{\mathrm{gd}} \simeq \Cat_\infty^p$$
of presentable real $\infty$-categories.
By \cref{cor:Ktheory_coincide} we have that $$\KR^{C_2}: 
\St^{\mathrm{gd}} \to \Sp_{\geq 0}$$ factors as
$$ \St^{\mathrm{gd}}\simeq \Cat_\infty^p\xrightarrow{\GW} \Sp_{\geq 0}.$$

Hence by \cref{cor_compa} and \cref{prop_Mackey} there is a commutative diagram

$$\begin{xy}
\xymatrix{& \St^\gd \ar[ld]_\simeq \ar[d]^{\rho_{\St^\gd}} \ar[rrr]^{\KR}
&&& \Sp^{C_2} \ar[d]^{\simeq}
\\
\infty\Cat^p \ar[rd]_{\rho_{\Cat_\infty^p}} & \Fun^{\prod}(\Span(\Set[C_2]),\St^\gd) \ar[d]^{\simeq}\ar[rrr]^{\KR^{C_2}_*} &&& \Fun^{\prod}(\Span(\Set[C_2]),\Sp)\ar[d]^{=}
\\
& \Fun^{\prod}(\Span(\Set[C_2]),\Cat_\infty^p) \ar[rrr]^{\GW_*} &&& \Fun^{\prod}(\Span(\Set[C_2]),\Sp).}
\end{xy}$$
\end{proof}

\bibliographystyle{alpha}
\bibliography{add}

\end{document}